\documentclass[12pt, reqno]{amsart}
\usepackage{}
\usepackage{stmaryrd}
\usepackage{amsmath}
\usepackage{amsfonts}
\usepackage{amssymb}
\usepackage[all,cmtip]{xy}           
\usepackage{xspace}
\usepackage{bbding}
\usepackage{txfonts}
\usepackage[shortlabels]{enumitem}
\usepackage{cite}

\usepackage{tikz}
\usepackage{titletoc}

\usepackage{ifpdf}
\ifpdf
 \usepackage[colorlinks,final,backref=page,hyperindex]{hyperref}
\else
 \usepackage[colorlinks,final,backref=page,hyperindex,hypertex]{hyperref}
\fi
\usepackage[active]{srcltx}

\makeatletter

\begin{document}
\newcommand {\emptycomment}[1]{} 

\baselineskip=15pt
\newcommand{\nc}{\newcommand}
\newcommand{\delete}[1]{}
\nc{\mfootnote}[1]{\footnote{#1}} 
\nc{\todo}[1]{\tred{To do:} #1}

\nc{\mlabel}[1]{\label{#1}}  
\nc{\mcite}[1]{\cite{#1}}  
\nc{\mref}[1]{\ref{#1}}  
\nc{\meqref}[1]{\eqref{#1}} 
\nc{\mbibitem}[1]{\bibitem{#1}} 

\delete{
\nc{\mlabel}[1]{\label{#1}  
{\hfill \hspace{1cm}{\bf{{\ }\hfill(#1)}}}}
\nc{\mcite}[1]{\cite{#1}{{\bf{{\ }(#1)}}}}  
\nc{\mref}[1]{\ref{#1}{{\bf{{\ }(#1)}}}}  
\nc{\meqref}[1]{\eqref{#1}{{\bf{{\ }(#1)}}}} 
\nc{\mbibitem}[1]{\bibitem[\bf #1]{#1}} 
}

\newcommand {\comment}[1]{{\marginpar{*}\scriptsize\textbf{Comments:} #1}}
\nc{\mrm}[1]{{\rm #1}}
\nc{\id}{\mrm{id}}  \nc{\Id}{\mrm{Id}}

\def\a{\alpha}
\def\b{\beta}
\def\bd{\boxdot}
\def\bbf{\bar{f}}
\def\bF{\bar{F}}
\def\bbF{\bar{\bar{F}}}
\def\bbbf{\bar{\bar{f}}}
\def\bg{\bar{g}}
\def\bG{\bar{G}}
\def\bbG{\bar{\bar{G}}}
\def\bbg{\bar{\bar{g}}}
\def\bT{\bar{T}}
\def\bt{\bar{t}}
\def\bbT{\bar{\bar{T}}}
\def\bbt{\bar{\bar{t}}}
\def\bR{\bar{R}}
\def\br{\bar{r}}
\def\bbR{\bar{\bar{R}}}
\def\bbr{\bar{\bar{r}}}
\def\bu{\bar{u}}
\def\bU{\bar{U}}
\def\bbU{\bar{\bar{U}}}
\def\bbu{\bar{\bar{u}}}
\def\bw{\bar{w}}
\def\bW{\bar{W}}
\def\bbW{\bar{\bar{W}}}
\def\bbw{\bar{\bar{w}}}
\def\btl{\blacktriangleright}
\def\btr{\blacktriangleleft}
\def\ci{\circ}
\def\d{\delta}
\def\dd{\diamondsuit}
\def\D{\Delta}
\def\G{\Gamma}
\def\g{\gamma}
\def\k{\kappa}
\def\l{\lambda}
\def\lr{\longrightarrow}
\def\o{\otimes}
\def\om{\omega}
\def\p{\psi}
\def\r{\rho}
\def\ra{\rightarrow}
\def\rh{\rightharpoonup}
\def\lh{\leftharpoonup}
\def\s{\sigma}
\def\st{\star}
\def\ti{\times}
\def\tl{\triangleright}
\def\tr{\triangleleft}
\def\v{\varepsilon}
\def\vp{\varphi}
\def\vth{\vartheta}

\newtheorem{thm}{Theorem}[section]
\newtheorem{lem}[thm]{Lemma}
\newtheorem{cor}[thm]{Corollary}
\newtheorem{pro}[thm]{Proposition}
\theoremstyle{definition}
\newtheorem{defi}[thm]{Definition}
\newtheorem{ex}[thm]{Example}
\newtheorem{rmk}[thm]{Remark}
\newtheorem{pdef}[thm]{Proposition-Definition}
\newtheorem{condition}[thm]{Condition}
\newtheorem{question}[thm]{Question}
\renewcommand{\labelenumi}{{\rm(\alph{enumi})}}
\renewcommand{\theenumi}{\alph{enumi}}

\nc{\ts}[1]{\textcolor{purple}{Tianshui:#1}}

\font\cyr=wncyr10

 \title[Infinitesimal (BiHom-)bialgebras of any weight]{\bf Infinitesimal (BiHom-)bialgebras of any weight (II): Representations}

 \author[Ma]{Tianshui Ma}
 \address{School of Mathematics and Information Science, Henan Normal University, Xinxiang 453007, China}
         \email{matianshui@htu.edu.cn}

 \author[Makhlouf]{Abdenacer Makhlouf\textsuperscript{*}}
 \address{Universit{\'e} de Haute Alsace, IRIMAS- D\'epartement de  Math{\'e}matiques,  18, rue des Fr{\`e}res Lumi{\`e}re F-68093 Mulhouse, France}

         \email{abdenacer.makhlouf@uha.fr}

  \thanks{\textsuperscript{*}Corresponding author}

\date{\today}

\begin{abstract} The aim of this paper is to investigate  representation theory of  infinitesimal (BiHom-)bialgebras of any weight $\l$ (abbr. $\l$-inf(BH)-bialgebras). Firstly, inspired by  the well-known Majid-Radford's bosonization theory in Hopf algebra theory, we present a class of $\l$-inf(BH)-bialgebras, named $\l$-inf(BH)-biproduct bialgebras, consisting of an inf(BH)-product algebra structure and an inf(BH)-coproduct coalgebra structure, which  induces a structure of a $\l$-inf(BH)-Hopf bimodule over a $\l$-inf(BH)-bialgebra. Secondly, we explore  relationships among $\l$-inf(BH)-Hopf bimodules, $\l$-Rota-Baxter (BiHom-)bimodules, (BiHom-)dendriform bimodules and (BiHom-)pre-Lie bimodules. Finally, we provide two kinds of general Gelfand-Dorfman theorems related to BiHom-Novikov algebras.
\end{abstract}

 \keywords{ $\l$-inf(BH)-bialgebra, representation, $\l$-infBH-Hopf bimodule, $\l$-inf(BH)-biproduct bialgebra, $\l$-Rota-Baxter (BiHom-)bimodule, (BiHom-)dendriform bimodule,  (BiHom-)pre-Lie bimodule}

\subjclass[2020]{
17B61,
17D30,
17B38,
17A30.}

 \maketitle

\tableofcontents

\numberwithin{equation}{section}
\allowdisplaybreaks

\section{Introduction and Preliminaries}
 Yetter-Drinfel'd modules (or Crossed (bi)modules, or Yang-Baxter modules) over a bialgebra were introduced by D. N. Yetter \cite{Yet}, they can be seen as a (co)modules over (dual) V. G. Drinfel'd quantum doubles \cite{PW,Ra12}. A bicovariant bimodule appeared as the basic notion, in S. L. Woronowicz's approach, to differential calculus on quantum groups \cite{Wor}. Both objects arise very often in Hopf algebras and quantum groups theories. A bialgebra $B$ in Hopf algebras theory is an algebra $(B,\cdot)$ and at the same time a coalgebra $(B, \D)$ satisfying the compatibility condition that  $\D$ is an algebra map. Replacing the compatibility condition by $\D$ is a derivation, then one obtains an infinitesimal bialgebras of weight 0 (abbr. 0-inf-bialgebras) introduced by S. A. Joni and G.-C. Rota in connection with the calculus of divided differences \cite{JR}, which is exactly an associative analog of Lie bialgebras presented by V. G. Drinfel'd \cite{Dr}. In \cite{Ag99}, M. Aguiar started to develop a theory for these objects analogous to the theory of ordinary Hopf algebras, see also his series of papers \cite{Ag00b,Ag04}. More results about  0-inf-bialgebras can be found\cite{Ag00b,Ag04,LMMP3,MMS,MY,MLiY,MYZZ,WW14,Yau10}. In \cite{Ag04}, M. Aguiar introduced a notion of infinitesimal Hopf bimodule of weight 0 (abbr. 0-inf-Hopf bimodule) also discussed theirs relations with bimodules over the Drinfel'd double for 0-inf-bialgebras, Rota-Baxter bimodules, Dendriform bimodules and pre-Lie bimodules. The notion of $(-1)$-inf-bialgebra \cite{LR06} was introduced by J.-L. Loday and M. O. Ronco in connection with a structure theorem of cofree Hopf algebras. K. Ebrahimi-Fard in his PhD thesis \cite{EF06} defined  $\l$-inf bialgebras (for any weight $\l$) unifying the above two cases. For further  studies related to $\l$-inf bialgebras in the literature, see for examples \cite{EF06,Fo10,LR06,ZCGL18,ZG,ZGZL}. It is worth mentioning that a unitary and counitary $0$-inf-bialgebra $H$ is trivial, i.e., $H=0$ and $\v(1)=0$ \cite{Ag99}, but for non-zero weight case, this result does not hold \cite{MM,ZG}. The following construction motivates the study of  infinitesimal bialgebras with non-zero weight: Let $(A, \mu, 1)$ be a unitary algebra, define a comultiplication $\D: A\lr A$ by $\D(a)=-\l(a\o 1)$ (resp. $ \D(a)=-\l(1\o a)$) for all $a\in A$. Then $(A, \mu, 1, \D)$ is a unitary $\l$-inf-bialgebra. Results on $\l$-inf-bialgebras in \cite{MM,MLi,ZG} show   that many structures rely on the weight $\l$ and provide a motivation   to enhance the study of  infinitesimal bialgebras with any weight.

 Majid-Radford's bosonization \cite{Ra,Ra12,Maj1,Maj2,Ma99} is a biproduct Hopf algebra including a smash product algebra structure and a smash coproduct coalgebra structure, which plays a central role in the lifting theory in the classification of finite dimension pointed Hopf algebras \cite{AS} and provides examples for Rota-Baxter bialgebras \cite{ML}, since it has a categorical interpretation \cite{Maj1,Maj2}: $B\star H$ is a biproduct bialgebra if and only if $B$ is a bialgebra in the braided monoidal category of (left-left) Yetter-Drinfel'd modules $_H^H{\mathcal{YD}}$. This provides us a motivation to introduce the concept of $\l$-inf-Hopf bimodule over a $\l$-inf-bialgebra, where $\l\in K$, based on the infinitesimal version of Majid-Radford's bosonization structures. Although coming from different intentions, $\l$-inf-Hopf bimodules here recover M. Aguiar's in \cite{Ag04} when $\l=0$.

 BiHom-type (co)associative algebras, introduced by G. Graziani, A. Makhlouf, C. Menini and F. Panaite in \cite{GMMP}, involving two commuting multiplicative linear maps, can be seen as an extension of Hom-type algebras which appeared first in quasi-deformations of Lie algebras of vector fields. Hom-Lie algebras were introduced  by J. T. Hartwig, D. Larsson and S. D. Silvestrov in \cite{HLS} and Hom-associative algebras by A. Makhlouf and S. D. Silvestrov in \cite{MS}. In 2020, the BiHom-type of $0$-inf-bialgebras were introduced  and studied by L. Liu, A. Makhlouf, C. Menini and F. Panaite \cite{LMMP3}, they generalize the Hom-case considered in \cite{Yau10}.  The basic definitions and properties of $\l$-infBH-bialgebras were introduced in \cite{MM} and the aim of this paper is to complete their study.

 In this paper, we introduce  infinitesimal Hopf bimodule over infinitesimal bialgebra of {\bf any weight} and also discuss  BiHom-deformation. In Theorem \ref{thm:infbipro}, we provide a construction of $\l$-infBH-biproduct bialgebra inducing the notion of $\l$-infBH-Hopf bimodule in Definition \ref{de:20.01}. Then we consider the relationships summarized  in the following diagram:
$$
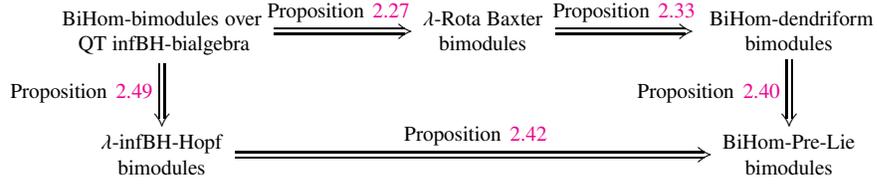
\hspace{15mm} \xymatrix@C4.3em{
&\text{BiHom-bimodules over}\atop \text{ QT infBH-bialgebra}
\ar@2{->}^{\quad\text{Proposition}\ \mref{pro:20.08}}[r]
\ar@2{->}_{\text{Proposition}\  \mref{pro:20.10}}[d]
&\text{$\l$-Rota Baxter}\atop \text{bimodules}
\ar@2{->}^{\text{Proposition}\ \mref{pro:20.07}\ \ \ }[r]
&\text{~~BiHom-dendriform}\atop \text{bimodules}\\
&\text{$\l$-infBH-Hopf}\atop \text{bimodules}
\ar@2{->}^{\text{Proposition}\  \mref{pro:020.012}}[rr]
&&\text{BiHom-Pre-Lie}\atop \text{bimodules}
\ar@2{<-}^{\text{Proposition}\  \mref{pro:20.05}}[u]
&&}
$$

 In last section, we investigate two kinds of general Gelfand-Dorfman theorems related to BiHom-Novikov algebras (see Theorems \ref{thm:20.15} and \ref{thm:B20.15}).

 Throughout this paper, $K$ will be a field, and all vector spaces, tensor products, and homomorphisms are over $K$. We use Sweedler's notation for terminology on coalgebras. For a coalgebra $C$, we write comultiplication $\D(c)=c_1\o c_2$, for any $c\in C$. 
 We denote by $\id_M$ the identity map from $M$ to $M$, $\tau: M\o N\ra N\o M$ the flip map. We abbreviate "infinitesimal BiHom-" to "infBH-", and "respectively" to "resp.".

 Let us recall from \cite{GMMP,MLi} the following basic definitions and structures.
 \begin{defi}\mlabel{de:1.1} A {\bf BiHom-associative algebra} is a 4-tuple $(A,\mu,\a,\b)$, where $A$ is a linear space, $\a,\b:A\lr A$ and $\mu:A\o A\lr A$ (write $\mu(a\o b)=ab$) are linear maps satisfying the following conditions, for all $a,b,c\in A$:
 \begin{eqnarray}
 &\a\circ\b=\b\circ\a,\quad \a(ab)=\a(a)\a(b),\quad\b(ab)=\b(a)\b(b),&\mlabel{eq:1.2}\\
 &\a(a)(bc)=(ab)\b(c).&\mlabel{eq:1.3}
 \end{eqnarray}
 \end{defi}
 A BiHom-associative algebra $(A,\mu,\a,\b)$ is called {\bf unitary} if there exists an element $1_{A}\in A$ (called a unit) such that
 \begin{eqnarray}
 &\a(1_{A})=1_{A},\quad\b(1_{A})=1_{A},\quad a1_{A}=\a(a),\quad 1_{A}a=\b(a),\quad \forall a\in A.&\mlabel{eq:1.5}
 \end{eqnarray}

 A morphism $f:(A,\mu_{A},\a_{A},\b_{A})\longrightarrow (B,\mu_{B},\a_{B},\b_{B})$ of BiHom-associative algebras is a linear map $f:A\longrightarrow B$ such that $\a_{B}\circ f=f\circ\a_{A}$, $\b_{B}\circ f=f\circ\b_{A}$ and $f\circ\mu_{A}=\mu_{B}\circ(f\o f)$.

 \begin{rmk}
 {\bf ``Yau twist"}: Let $(A, \mu)$ be an associative algebra, $\a, \b: A\lr A$ two linear maps satisfying Eq.(\mref{eq:1.2}). Then $(A, \mu\ci (\a\o \b), \a, \b)$ is a BiHom-associative algebra.
 \end{rmk}

 \begin{defi}\mlabel{de:1.2} A {\bf BiHom-coassociative coalgebra} is a 4-tuple $(C,\D,\psi,\om)$, in which $C$ is a linear space, $\psi,\om:C\lr C$ and $\D:C\lr C\o C$ are linear maps, such that
 \begin{eqnarray}
 &\psi\circ\om=\om\circ \psi,~~~(\psi\o\psi)\circ\D=\D\circ\psi,~~
 (\om\o\om)\circ\D=\D\circ\om,&\mlabel{eq:1.7}\\
 &(\D\o\psi)\circ\D=(\om\o \D)\circ\D.&\mlabel{eq:1.9}
 \end{eqnarray}
 \end{defi}

 A BiHom-coassociative coalgebra $(C,\D,\psi,\om)$ is called {\bf counitary} if there exists a linear map $\v: C\lr K$ (called a counit) such that
 \begin{eqnarray}
 &\v\circ\psi=\v,\quad \v\circ\om=\v,\quad (\id_{C}\o\v)\circ\D=\om,\quad(\v\o \id_{C})\circ\D=\psi.&\mlabel{eq:1.11}
 \end{eqnarray}

 A morphism $g:(C,\D_{C},\psi_{C},\om_{C})\longrightarrow (D,\D_{D},\psi_{D},\om_{D})$ of BiHom-coassociative coalgebras is a linear map $g:C\longrightarrow D$ such that $\om_{D}\circ g=g\circ\om_{C}$, $\psi_{D}\circ g=g\circ\psi_{C}$ and $\D_{D}\circ g= (g\o g)\circ\D_{C}$.

 \begin{defi}\mlabel{de:1.3} Let $(A,\mu_A,\a_{A},\b_{A})$ be a BiHom-associative algebra. A {\bf left $(A,\mu_A,\a_{A},\b_{A})$-module} is a 4-tuple $(M,\tl,\a_{M},\b_{M})$,  where $M$ is a linear space,  $\a_{M}, \b_{M}: M\lr M$ and $\tl: A\o M\lr M$ (write $\tl(a\o m)=a\tl m$) are linear maps such that, for all $a, a'\in A, m\in M$,
 \begin{eqnarray}
 &\a_{M}\circ\b_{M}=\b_{M}\circ\a_{M},~~\a_{M}(a\triangleright m)=\a_{A}(a)\tl\a_{M}(m),~~\b_{M}(a\tl m)=\b_{A}(a)\tl\b_{M}(m),&\mlabel{eq:1.13}\\
 &\a_{A}(a)\tl(a'\tl m)=(aa')\tl\b_{M}(m).&\mlabel{eq:1.15}
 \end{eqnarray}
 Likewise, we can get the right version of $(A,\mu_A,\a_{A},\b_{A})$-module.

 If $(M,\tl,\a_{M},\b_{M})$ is a left $(A,\mu_A,\a_{A},\b_{A})$-module and at the same time $(M,\tr,\a_{M},\b_{M})$ is a right $(A,\mu_A,\a_{A},\b_{A})$-module (write $\tr(m\o a)=m\tr a$), then $(M,\tl,\tr,\a_{M},\b_{M})$ is an {\bf $(A,\mu_A,\a_{A},\b_{A})$-bimodule} if
 \begin{eqnarray}
 &\a_{A}(a)\tl(m\tr a')=(a\tl m)\tr \b_{A}(a').&\mlabel{eq:1.16}
 \end{eqnarray}

 Furthermore, if $(M, \mu_M, \a_M, \b_M)$ is a BiHom-associative algebra such that
 \begin{eqnarray}
 &\a_A(a)\tl (m n)=(a\tl m)\b_M(n),&\mlabel{eq:bimoda1}\\
 &\a_M(m)(n \tr a)=(m n)\tr \b_A(a),&\mlabel{eq:bimoda2}\\
 &\a_M(m)(a\tl n)=(m\tr a)\b_M(n)&\mlabel{eq:bimoda3}
 \end{eqnarray}
 hold for all $a\in A$ and $m, n\in M$, then we call $(M, \mu_M, \g, \nu, \a_{M}, \b_{M})$ is an {\bf $(A,\mu,\a_{A},\b_{A})$-bimodule algebra}.
 \end{defi}

 Dually, we have the following definitions.

 \begin{defi}\mlabel{de:01.14} Let $(C,\D_{C},\psi_{C},\om_{C})$ be a BiHom-coassociative coalgebra. A {\bf left $(C, \D_C,$ $\psi_{C},\om_{C})$-comodule} is a 4-tuple $(M, \rho, \psi_{M},\om_{M})$, where $M$ is a linear space, $\psi_{M},\om_{M}: M\rightarrow M$, $\rho: M\rightarrow C\o M$ are linear maps
 such that the following conditions are satisfied,
 \begin{eqnarray}
 &\psi_{M}\circ \om_{M}=\om_{M}\circ \psi_{M},~~(\psi_{C}\o \psi_{M})\circ \rho=\rho\circ \psi_{M},~~ (\om_{C}\o \om_{M})\circ \rho=\rho\circ \om_{M},& \mlabel{eq:01.22a}\\
 &(\om_{C}\o \rho)\circ \rho=(\D_{C}\o \psi_{M})\circ \rho.&\mlabel{eq:01.24}
 \end{eqnarray}
 If $(M,\rho,\psi_{M},\om_{M})$ is a left $(C,\D_C,\psi_{C},\om_{C})$-comodule and at the same time $(M,\vp,\psi_{M},\om_{M})$ is a right $(C,\D_C,\psi_{C},\om_{C})$-comodule, then $(M,\rho,\vp,\psi_{M},\om_{M})$ is a {\bf $(C,\D_C,\psi_{C},\om_{C})$-bicomodule} if
 \begin{eqnarray}
 &(\om_{C}\o \vp)\circ\rho=(\rho\o \psi_{C})\circ \vp.& \mlabel{eq:01.24aa}
 \end{eqnarray}

 Furthermore, if $(M, \D_M, \psi_M, \om_M)$ is a BiHom-coassociative coalgebra such that
 \begin{eqnarray}
 &(\om_A\o \D_M)\rho=(\rho\o \psi_M)\D_M,&\mlabel{eq:bicmodca1}\\
 &(\om_M\o \vp)\D_M=(\D_M\o \psi_A)\vp,&\mlabel{eq:bicmodca2}\\
 &(\om_M\o \rho)\D_M=(\vp\o \psi_M)\D_M&\mlabel{eq:bicmodca3}
 \end{eqnarray}
 hold, then we call $(M, \D_M, \rho, \vp, \psi_{M}, \om_{M})$ is a {\bf $(C,\D,\psi_{A},\om_{A})$-bicomodule coalgebra}.
 \end{defi}

\section{$\l$-infBH-Hopf bimodules}\mlabel{se:bimodule}
 The notion of $\l$-Rota-Baxter bimodule was introduced in \cite{BGM} by C. M. Bai, L. Guo and T. S. Ma in order to establish the bialgebra theory of Rota-Baxter algebras. In this section, we introduce  the notion of $\l$-infBH-Hopf bimodule which proposes the compatibility conditions for left and right $\l$-infBH-modules. We establish a $\l$-infBH-biproduct bialgebra which leads to $\l$-infBH-Hopf bimodules. We also discuss the relationship with Rota-Baxter BiHom-bimodules of any weight, BiHom-pre-Lie bimodules and BiHom-dendriform bimodules.

\subsection{Basic definitions} We first recall some basic definitions from \cite{MM}.
 \begin{defi}\mlabel{de:12.1} Let $\l$ be a given element of $K$. (1) A {\bf $\l$-infinitesimal BiHom-bialgebra} (abbr. $\l$-infBH-bialgebra) is a 7-tuple $(A, \mu, \D, \a, \b, \psi, \om)$ with the property that $(A, \mu, \a, \b)$ is a BiHom-associative algebra, $(A, \D, \psi, \om)$ is a BiHom-coassociative coalgebra satisfying the following conditions
 \begin{eqnarray}
 &\a\circ\psi=\psi\circ\a,\ \a\circ\om=\om\circ\a,\ \b\circ\psi=\psi\circ\b,\ \b\circ\om=\om\circ\b,&\mlabel{eq:12.1}\\
 &(\a\o\a)\circ\D=\D\circ\a,\ (\b\o\b)\circ\D=\D\circ\b,&\mlabel{eq:12.2}\\
 &\psi\circ\mu=\mu\circ(\psi\o\psi),\ \om\circ\mu=\mu\circ(\om\o\om),&\mlabel{eq:12.3}\\
 &\D\circ\mu=(\mu\o\b)\circ(\om\o\D)
 +(\a\o\mu)\circ(\D\o\psi)+\l(\a\om\o\b\psi).&\mlabel{eq:12.4}
 \end{eqnarray}

 (2) If further $(A, \mu, 1, \a, \b)$ is a unitary BiHom-associative algebra, then the 8-tuple $(A, \mu, 1, \D,$ $\a, \b, \psi, \om)$ is called a {\bf unitary $\l$-infBH-bialgebra} if
 \begin{eqnarray}
 & \psi(1_{A})=1_{A},\ \om(1_{A})=1_{A}.&\mlabel{eq:12.30}
 \end{eqnarray}

 (3) If further $(A, \D, \v, \psi, \om)$ is a counitary BiHom-coassociative coalgebra, then the 8-tuple $(A, \mu, \D, \v, \a, \b, \psi, \om)$ is called a {\bf counitary $\l$-infBH-bialgebra} if
 \begin{eqnarray}
 &\v\circ\a=\v,\ \v\circ\b=\v.&\mlabel{eq:12.31}
 \end{eqnarray}
 \end{defi}

 \begin{rmk}\mlabel{rmk:12.2} (1) If $\l=0$ in Definition \mref{de:12.1}, then we get infBH-bialgebras introduced by Liu, Makhlouf, Menini, Panaite in \cite[Definition 4.1]{LMMP3} and also studied in \cite{MLi,MLiY,MY}. If further $\a=\b=\psi=\om=\id$, then one  obtains Joni and Rota's infinitesimal bialgebras \mcite{JR}.

 (2) If $\l=-1$ in Definition \mref{de:12.1}, then we have the BiHom-version of Loday and Ronco's infinitesimal bialgebras \cite{LR06}.

 (3) Definition \mref{de:12.1} is the BiHom-version of Ebrahimi-Fard's $\l$-infinitesimal bialgebras \cite{EF06} studied in \cite{ZG}.

 (4) A morphism between $\l$-infBH-bialgebras is a linear map that commutes with the structure maps $\a, \b, \psi, \om$, the multiplication $\mu$ and the comultiplication $\D$.
 \end{rmk}

 \begin{defi}\mlabel{de:12.17} Let $\l$ be a given element of $K$ and $(A,\mu,\D,\a_{A},\b_{A},\psi_{A},\om_{A})$ be a $\l$-infBH-bialgebra. A {\bf left $\l$-infBH-Hopf module} is a 7-tuple $(M,\tl,\rho, \a_{M},\b_{M},$ $\psi_{M},\om_{M})$, where $(M, \tl, \a_{M}$, $\b_{M})$ is a left $(A,\mu,\a_{A},\b_{A})$-module and $(M, \rho, \psi_{M}, \om_{M})$ is a left $(A,\D,\psi_{A},\om_{A})$-comodule, such that any two maps of $\a_{M}, \b_{M}, \psi_{M}, \om_{M}$ commute and
 \begin{eqnarray}
 &\rho\circ \tl=(\mu\o \b_{M})(\om_{A}\o\rho)+(\a_{A}\o\tl)(\D\o\psi_{M})
 +\l\a_{A}\om_{A}\o\b_{M}\psi_{M}.&\mlabel{eq:12.13}
 \end{eqnarray}

 If further $(A,\mu,1,\D,\a_{A},\b_{A},\psi_{A},\om_{A})$ ($(A,\mu,\v,\D,\a_{A},\b_{A},\psi_{A},\om_{A})$) is a (co)unitary $\l$-infBH-bialgebra, then $(M,\tl,\rho,$ $\a_{M},\b_{M},\psi_{M},\om_{M})$ is called a {\bf (co)unitary left $\l$-infBH-Hopf module}.
 \end{defi}

 \begin{rmk}\mlabel{rmk:12.18} In terms of above notations, the compatibility condition of a left $\l$-infBH-Hopf module given in Eq.(\mref{eq:12.13}) may be written as
 \begin{eqnarray*}
 &(a\tl m)_{(-1)}\o (a\tl m)_{(0)}=\om_{A}(a)m_{(-1)}\o\b_{M}(m_{(0)})+\a_{A}(a_{1})\o a_{2}\tl\psi_{M}(m)+
 \l\a_{A}\om_{A}(a)\o\b_{M}\psi_{M}(m),&
 \end{eqnarray*}
 here we write $\tl(a\o m)=a\tl m$ and $\rho(m)=m_{(-1)}\o m_{(0)}$.
 \end{rmk}
 
 \begin{defi}\mlabel{de:12.17a} Let $\l$ be a given element of $K$ and $(A,\mu,\D,\a_{A},\b_{A},\psi_{A},\om_{A})$ a $\l$-infBH-bialgebra. A {\bf right $\l$-infBH-Hopf module} is a 7-tuple $(M,\tr,\varphi, \a_{M},\b_{M},$ $\psi_{M},\om_{M})$, where $(M, \tr, \a_{M}, \b_{M})$ is a right $(A,\mu,\a_{A},\b_{A})$-module and $(M, \varphi, \psi_{M}, \om_{M})$ is a right $(A,\D,\psi_{A},\om_{A})$-comodule, such that any two maps of $\a_{M}, \b_{M}, \psi_{M}, \om_{M}$ commute and
 \begin{eqnarray}
 &\varphi\circ \tr= (\tr\o \b_{A})(\om_{M}\o\D)+(\a_{M}\o\mu)(\varphi\o\psi)
 +\l\a_{M}\om_{M}\o\b_{A}\psi_{A}.&\mlabel{eq:12.13a}
 \end{eqnarray}

 If further $(A,\mu,1,\D,\a_{A},\b_{A},\psi_{A},\om_{A})$ ($(A,\mu,\v,\D,\a_{A},\b_{A},\psi_{A},\om_{A})$) is a (co)unitary $\l$-infBH-bialgebra, then $(M,\tr,\varphi,$ $\a_{M},\b_{M},\psi_{M},\om_{M})$ is called a {\bf (co)unitary right $\l$-infBH-Hopf module}.
 \end{defi}
 
 \begin{rmk}\mlabel{rmk:12.18a} In terms of above notations, the compatibility condition of a right $\l$-infBH-Hopf module given in Eq.(\mref{eq:12.13a}) may be written as
 \begin{eqnarray*}
 &(m\tr a)_{(0)}\o (m\tr a)_{(1)}=\om_{M}(m)\tr a_{1}\o \b_{A}(a_{2})+\a_{M}(m_{(0)})\o m_{(1)}\psi_{A}(a)+\l\a_{M}\om_{M}(m)\o \b_{A}\psi_{A}(a),&
 \end{eqnarray*}
 here we write $\tr(m\o a)=m\tr a$ and $\vp(m)=m_{(0)}\o m_{(1)}$.
 \end{rmk}
 
 \begin{defi}\mlabel{de:20.01} Let $(A, \mu, \D, \a_{A}, \b_{A}, \psi_{A}, \om_{A})$ be a $\l$-infBH-bialgebra, $M$ a vector space, $\a_{M}$, $\b_{M}$, $\psi_{M}$, $\om_{M}:\ M\longrightarrow M$ four linear maps such that any two of them commute. A {\bf $\l$-infBH-Hopf bimodule} over  $(A,\mu,\D,\a_{A},\b_{A},\psi_{A},\om_{A})$ is a 9-tuple $(M,\tl,\tr,\rho,\varphi,\a_{M},\b_{M},\psi_{M},\om_{M})$, where
 \begin{eqnarray*}
 \tl:A\o M\longrightarrow M,\ \tr:M\o A\longrightarrow M,\ \rho:M\longrightarrow A\o M\ \hbox{and}\ \varphi:M\longrightarrow M\o A
 \end{eqnarray*}
 are linear maps satisfying the following conditions:

 (1) $(M,\tl,\rho,\a_{M},\b_{M},\psi_{M},\om_{M})$ is a left $\l$-infBH-Hopf module over $(A,\mu,\D,\a_{A},\b_{A},\psi_{A},\om_{A})$.

 (2) $(M,\tr,\varphi,\a_{M},\b_{M},\psi_{M},\om_{M})$ is a right $\l$-infBH-Hopf module over $(A,\mu,\D,\a_{A},\b_{A},\psi_{A},\om_{A})$.

 (3) $(M,\tl,\tr,\a_{M},\b_{M})$ is a BiHom-bimodule over $(A,\mu, \a_{A},\b_{A})$.

 (4) $(M,\rho,\varphi,\psi_{M},\om_{M})$ is a BiHom-bicomodule over $(A, \D, \psi_{A},\om_{A})$.

 (5) the following equations hold:
 \begin{eqnarray}
 &(\tl\o\b_{A})(\om_{A}\o\varphi)=\varphi\circ\tl,&\mlabel{eq:20.01}\\
 &(\a_{A}\o\tr)(\rho\o\psi_{A})=\rho\circ\tr.&\mlabel{eq:20.02}
 \end{eqnarray}
 \end{defi}

 \begin{rmk} (1) Eq.(\mref{eq:20.01}) (resp. Eq.(\mref{eq:20.02})) means that $(M, \tl, \vp,\a_{M},\b_{M},\psi_{M},\om_{M})$ (resp. $(M, \tr, \rho$, $\a_{M}, \b_{M}, \psi_{M}, \om_{M})$) is a left-right (resp. right-left) BiHom-Long module. As we know that Long module was introduced by Long in \cite{Long} to develop the Brauer group theory for algebras.

 (2) If the structure maps $\a_{A}=\b_{A}=\psi_{A}=\om_{A}=\id_A$ and $\a_{M}=\b_{M}=\psi_{M}=\om_{M}=\id_M$ in Definition \mref{de:20.01}, then the 5-tuple $(M,\tl,\tr,\rho,\varphi)$ is called a {\bf $\l$-inf-Hopf bimodule} over the $\l$-inf-bialgebra $(A,\mu,\D)$.

 \end{rmk}
 
 \begin{ex}\mlabel{ex:20.02} For any $\l$-infBH-bialgebra $(A, \mu, \D, \a, \b,\psi,\om)$, the space $M=A\o A$ may be endowed with the following $\l$-infBH-Hopf $(A, \mu, \D, \a, \b,\psi,\om)$-bimodule structure:
 \begin{eqnarray*}
 \gamma=\mu\o\b_{A},\ \nu=\a_{A}\o\mu,\ \rho=\D\o\psi_{A}\ and\ \varphi=\om_{A}\o\D.
 \end{eqnarray*}
 \end{ex}
 
 \begin{thm}\mlabel{thm:20.13} Let $(M, \tl, \tr, \rho, \varphi)$ be a $\l$-inf-Hopf bimodule over a $\l$-inf-bialgebra $(A, \mu, \D)$, $\a_A$, $\b_A$, $\psi_A$, $\om_A: A\longrightarrow A$ be morphisms of algebras and coalgebras such that any two of them commute, $\a_{M}$, $\b_{M}$, $\psi_{M}$, $\om_{M}:M\longrightarrow M$ four linear maps such that any two of them commute and
 \begin{eqnarray*}
 &\a_{M}(a\tl m)=\a_{A}(a)\tl\a_{M}(m),\quad
 \a_{M}(m)_{(-1)}\o\a_{M}(m)_{(0)}=\a_{A}(m_{(-1)})\o\a_{M}(m_{(0)});&\\
 &\b_{M}(a\tl m)=\b_{A}(a)\tl\b_{M}(m),\quad
 \b_{M}(m)_{(-1)}\o\b_{M}(m)_{(0)}=\b_{A}(m_{(-1)})\o\b_{M}(m_{(0)});&\\
 &\psi_{M}(a\tl m)=\psi_{A}(a)\tl\psi_{M}(m),\quad
 \psi_{M}(m)_{(-1)}\o\psi_{M}(m)_{(0)}=\psi_{A}(m_{(-1)})\o\psi_{M}(m_{(0)});&\\
 &\om_{M}(a\tl m)=\om_{A}(a)\tl\om_{M}(m),\quad
 \om_{M}(m)_{(-1)}\o\om_{M}(m)_{(0)}=\om_{A}(m_{(-1)})\o\om_{M}(m_{(0)});&\\
  &\a_{M}(m\tr a)=\a_{M}(m)\tr\a_{A}(a),\quad
 \a_{M}(m)_{(0)}\o\a_{M}(m)_{(1)}=\a_{M}(m_{(0)})\o\a_{A}(m_{(1)});&\\
 &\b_{M}(m\tr a)=\b_{M}(m)\tr\b_{A}(a),\quad
 \b_{M}(m)_{(0)}\o\b_{M}(m)_{(1)}=\b_{M}(m_{(0)})\o\b_{A}(m_{(1)});&\\
 &\psi_{M}(m\tr a)=\psi_{M}(m)\tr\psi_{A}(a),\quad
 \psi_{M}(m)_{(0)}\o\psi_{M}(m)_{(1)}=\psi_{M}(m_{(0)})\o\psi_{A}(m_{(1)});&\\
 &\om_{M}(m\tr a)=\om_{M}(m)\tr\om_{A}(a),\quad
 \om_{M}(m)_{(0)}\o\om_{M}(m)_{(1)}=\om_{M}(m_{(0)})\o\om_{A}(m_{(1)}).&
 \end{eqnarray*}
 Then $M_{(\a, \b, \psi, \om)}: =(M, \tl_{(\a,\b)}: =\tl\circ(\a_{A}\o\b_{M}), \tr_{(\a,\b)}: =\tr\circ(\a_{M}\o\b_{A}), \rho_{(\psi,\om)}: =(\om_{A}\o\psi_{M})\circ\rho, \varphi_{(\psi,\om)}: =(\om_{M}\o\psi_{A})\circ\varphi, \a_{M}, \b_{M}, \psi_{M}, \om_{M})$ is a $\l$-infBH-Hopf bimodule over $(A, \mu_{(\a,\b)}:=\mu\circ (\a_A\o \b_A), \D_{(\psi, \om)}:=(\psi_A\o \om_A)\circ \D, \a_{A}, \b_{A}, \psi_{A}, \om_{A})$, called the Yau twist of $(M,\tl,\tr,\rho,\varphi)$.
 \end{thm}

 \begin{proof} By \cite[Theorem 2.3]{MM}, we  obtain that $A_{(\a,\b,\psi,\om)}:=(A,\mu_{(\a_A,\b_A)}:
 =\mu\circ(\a_A\o\b_A),
 \D_{(\psi_A,\om_A)}:=(\om_A\o\psi_A)\circ\D,\a_A,\b_A,\psi_A,\om_A)$ is a $\l$-infBH-bialgebra. For simplicity, we denote $\tl_{(\a,\b)}(a\o m)=a\tl' m=\a_{A}(a)\tl \b_{M}(m)$, $\tr_{(\a,\b)}(m\o a)=m\tr' a=\a_{M}(m)\tl \b_{A}(a)$, $\rho_{(\psi,\om)}(m)=m_{[-1]}\o m_{[0]}=\om_{A}(m_{(-1)})\o\psi_{M}(m_{(0)})$ and $\varphi_{(\psi,\om)}(m)=m_{[(0)]}\o m_{[(1)]}=\om_{M}(m_{(0)})\o\psi_{A}(m_{(1)})$ for all $a\in A$ and $m\in M$. The fact that $\a_{M}(a\tl' m)=\a_{A}(a)\tl'\a_{M}(m)$, $\b_{M}(a\tl' m)=\b_{A}(a)\tl'\b_{M}(m)$, $\a_{A}(a)\tl'(b\tl' m)=(ab)\tl'\b_{M}(m)$, $\a_{A}(a)\tl'(m\tr' b)=(a \tl'm)\tr'\b_{A}(b)$ hold obviously. We show that  (1) and (3) in Definition \mref{de:20.01} hold as follows:
 \begin{eqnarray*}
 &&\om_{A}(a)\ast m_{[-1]}\o\b_{M}(m_{[0]})+\a_{A}(a_{[1]})\o a_{[2]}\tl'\psi_{M}(m)+ \l\a_{A}\om_{A}(a)\o \b_{M}\psi_{M}(m)\\
 &&\hspace{10mm}=\a_{A}\om_{A}(a)\b_{A}\om_{A}(m_{(-1)})\o\b_{M}\psi_{M}(m_{(0)})
 +\a_{A}\om_{A}(a_{1})\o \a_{A}\psi_{A}(a_{2})\tl\b_{M}\psi_{M}(m)\\
 &&\hspace{15mm}+\l\a_{A}\om_{A}(a)\o \b_{A}\psi_{A}(b)\\
 &&\hspace{10mm}=(\om_{A}\o\psi_{M})(\a_{A}(a)\b_{A}(m_{(-1)})\o\b_{M}(m_{(0)})
 +\a_{A}(a_{1})\o \a_{A}(a_{2})\tl\b_{M}(m)\\
 &&\hspace{15mm}+ \l\a_{A}(a)\o \b_{M}(m))\\
 &&\hspace{10mm}=\om_{A}((\a_{A}(a)\tl\b_{M}(m))_{(-1)})\o \psi_{M}((\a_{A}(a)\tl\b_{M}(m))_{(0)})\\
 &&\hspace{10mm}=(a\tl' m)_{[-1]}\o (a\tl' m)_{[0]}.
 \end{eqnarray*}
 Therefore Eq.(\mref{eq:12.13}) holds, the rest can be checked immediately. Similarly, we  obtain (2) and (4) in Definition \mref{de:20.01}. Next we only need to prove that Eqs.(\mref{eq:20.01}) and (\mref{eq:20.03}) hold as follows,
 \begin{eqnarray*}
 \om_{A}(a)\tl' m_{[(0)]}\o\b_{A}(m_{[(1)]})
 &=&\a_{A}\om_{A}(a)\tl\b_{M}\om_{M}(m_{(0)})\o\b_{A}\psi_{A}(m_{(1)})\\
 &=&\om_{M}(\a_{A}(a)\tl\b_{M}(m_{(0)}))\o\psi_{A}\b_{A}(m_{(1)})\\
 &=&\om_{M}((\a_{A}(a)\tl\b_{M}(m))_{(0)})\o \psi_{A}((\a_{A}(a)\tl\b_{M}(m))_{(1)})\\
 &=&(a\tl' m)_{[(0)]}\o (a\tl' m)_{[(1)]},
 \end{eqnarray*}
 and
 \begin{eqnarray*}
 \a_{A}(m_{[-1]})\o m_{[0]}\tr'\psi_{A}(a)
 &=&\a_{A}\om_{A}(m_{(-1)})\o\a_{M}\psi_{M}(m_{(0)})\tr \psi_{A}\b_{A}(a)\\
 &=&\om_{A}((\a_{M}(m)\tr\b_{A}(a))_{(-1)})\o
 \psi_{M}((\a_{M}(m)\tr\b_{A}(a))_{(0)})\\
 &=&(m \tr' a)_{[-1]}\o (m \tr' a)_{[0]},
 \end{eqnarray*}
 finishing the proof.
 \end{proof}

 \subsection{Infinitesimal BiHom-biproduct bialgebras of weight $\l$}
 In what follows, we provide an explanation for $\l$-infBH-Hopf bimodules based on the construction of a class of infinitesimal BiHom-bialgebras, named $\l$-infBH-biproduct bialgebras, following the well-known Majid-Radford bosonization in  Hopf algebras theory.

 \begin{lem}\label{lem:smashproduct} Let $A, M$ be two vector spaces together with two commuting linear maps $\a_A, \b_A:A\lr A$, two commuting linear maps $\a_M, \b_M:M\lr M$, four linear maps $\mu_A: A\o A\lr A, a\o b\mapsto a b$, $\mu_M:M\o M\lr M, m\o n\mapsto m n$, $\tl:A\o M\lr M, a\o m\mapsto a\tl m$ and  $\tr:M\o A\lr M, m\o a\mapsto m\tr a$. Define a binary operation $\mu_{A\oplus M}$ on $A\oplus M$ by
 \begin{eqnarray}\label{eq:smashproduct}
 (a+m)(b+n):=a b+a\tl n+m\tr b +m n,\quad \forall~a, b\in A,~m, n\in M.
 \end{eqnarray}
 Then $(A\oplus M, \mu_{A\oplus M}, \a_{A}\oplus\a_{M},\b_{A}\oplus\b_{M})$, where $(\a_{A}\oplus\a_{M})(a+m)=\a_{A}(a)+\a_{M}(m)$, $(\b_{A}\oplus\b_{M})(a+m)=\b_{A}(a)+\b_{M}(m)$, is a BiHom-associative algebra if and only if
 \begin{enumerate}
   \item $(A, \mu_A, \a_A, \b_A)$ is a BiHom-associative algebra;
   \item $(M, \mu_M, \tl, \tr, \a_M, \b_M)$ is a BiHom-bimodule algebra over $(A, \mu_A, \a_A, \b_A)$.
 \end{enumerate}
 In this case, we call this algebra an {\bf infBH-product algebra of $A$ and $M$}, denoted by $(A\# M, \a_{A}\oplus\a_{M}, \b_{A}\oplus\b_{M})$.
 \end{lem}

 \begin{proof} For all $a, b, c\in A$ and $m, n, p\in M$, we calculate
 \begin{eqnarray*}
 (\a_A(a)+\a_M(m))((b+n)(c+p))
 &=&\a_A(a)(bc)+\a_A(a)\tl(b\tl p)+\a_A(a)\tl(n \tr c)\\
 &&+\a_A(a)\tl(n p)+\a_M(m)\tr (b c)+\a_M(m)(b\tl p)\\
 &&+\a_M(m)(n\tr c)+\a_M(m)(n p)
 \end{eqnarray*}
 and
 \begin{eqnarray*}
 ((a+m)(b+n))(\b_A(c)+\b_M(p))
 &=&(a b) \b_A(c)+(a b)\tl \b_M(p)+(a\tl n)\tr \b_A(c)\\
 &&+(a\tl n)\b_M(p)+(m\tr b)\tr \b_A(c)+(m\tr b)\b_M(p)\\
 &&+(m n)\tr \b_A(c)+(m n) \b_M(p).
 \end{eqnarray*}
 The rest is straitforward.
 \end{proof}

 \begin{cor}\label{cor:bimodalg} With notations in Lemma \ref{lem:smashproduct}, let $(A, \mu_A, \a_A, \b_A)$ be a BiHom-associative algebra. Then $(A\# M, \a_{A}\oplus\a_{M}, \b_{A}\oplus\b_{M})$ is an infBH-product algebra if and only if $(M, \mu_M, \tl, \tr, \a_M, \b_M)$ is a BiHom-bimodule algebra over $(A, \mu_A, \a_A, \b_A)$.
 \end{cor}

 \begin{lem}\label{lem:smashcoproduct} Let $C, M$ be two vector spaces together with two commuting linear maps $\psi_C, \om_C:C\lr C$, two commuting linear maps $\psi_M, \om_M:M\lr M$, four linear maps $\D_C: C\lr C\o C, c\mapsto c_1\o c_2$, $\D_M: M\lr M\o M, m\mapsto m_1\o m_2$, $\rho: M\lr C\o M, m\mapsto m_{(-1)}\o m_{(0)}$ and $\vp: M\lr M\o C, m\mapsto m_{(0)}\o m_{(1)}$. Define a linear operation $\D_{C\oplus M}$ on $C\oplus M$ by
 \begin{eqnarray}\label{eq:smashcoproduct}
 \D_{C\oplus M}(c+m):=c_1\o c_2+m_{(-1)}\o m_{(0)}+m_{(0)}\o m_{(1)}+m_1\o m_2,\quad \forall~c\in C,~m\in M.
 \end{eqnarray}
 Then $(C\oplus M, \D_{C\oplus M}, \psi_{C}\oplus\psi_{M},\om_{C}\oplus\om_{M})$, where $(\psi_{C}\oplus\psi_{M})(a+m)=\psi_{C}(a)+\psi_{M}(m)$, $(\om_{C}\oplus\om_{M})(a+m)=\om_{C}(a)+\om_{M}(m)$, is a BiHom-coassociative coalgebra if and only if
 \begin{enumerate}
   \item $(C, \D_C, \psi_C, \om_C)$ is a BiHom-coassociative coalgebra;
   \item $(M, \D_M, \rho, \vp, \psi_M, \om_M)$ is a BiHom-bicomodule coalgebra over $(C, \D_C, \psi_C, \om_C)$.
 \end{enumerate}
 In this case, we call this coalgebra an {\bf infBH-coproduct coalgebra of $C$ and $M$}, denoted by $(C\times M, \psi_{C}\oplus\psi_{M}, \om_{C}\oplus\om_{M})$.
 \end{lem}

 \begin{proof} We omit the proof since it is dual to Lemma \ref{lem:smashproduct}.
 \end{proof}

 \begin{cor}\label{cor:bicomodcoalg} With notations in Lemma \ref{lem:smashcoproduct}, let $(C, \D_C, \psi_C, \om_C)$ be a BiHom-coassociative coalgebra. Then $(C\times M, \psi_{C}\oplus\psi_{M}, \om_{C}\oplus\om_{M})$ is an infBH-coproduct coalgebra if and only if $(M, \D_M, \rho, \vp, \psi_M, \om_M)$ is a BiHom-bicomodule coalgebra over $(C, \D_C, \psi_C, \om_C)$.
 \end{cor}
 
 Similar to  Hopf algebras theory, we  introduce below some new notions.
 
 \begin{defi}\label{de:bicomalg} Let $(A, \mu, \D, \a_{A}, \b_{A}, \psi_{A}, \om_{A})$ be a $\l$-infBH-bialgebra, $(M, \tl, \tr, \a_M, \b_M)$ be a BiHom-bimodule over $(A, \mu_A, \a_A, \b_A)$ and $(M, \D_M, \psi_M, \om_M)$ be a BiHom-coassociative coalgebra. The 8-tuple $(M, \D_M, \tl, \tr, \a_M, \b_M, \psi_M, \om_M)$ is called a {\bf BiHom-module coalgebra} if any two of $\a_M, \b_M, \psi_M, \om_M$ commute and the following conditions hold:
 \begin{eqnarray}
 &\D_M\circ \a_M=(\a_M\o \a_M)\circ \D_M,~~\D_M\circ \b_M=(\b_M\o \b_M)\circ \D_M,&\label{eq:modcoalg1}\\
 &\psi_M(a\tl m)=\psi_A(a)\tl \psi_M(m),~~\om_M(a\tl m)=\om_A(a)\tl \om_M(m),&\label{eq:modcoalg2}\\
 &\psi_M(m\tr a)=\psi_M(a)\tr \psi_A(a),~~\om_M(m\tl a)=\om_M(m)\tr \om_A(a),&\label{eq:modcoalg3}\\
 &\D_M(a\tl m)=\om_{A}(a)\tl m_{1}\o\b_{M}(m_{2}),&\label{eq:modcoalg4}\\
 &\D_M(m\tr a)=\a_{M}(m_{1})\o m_{2}\tr\psi_{A}(b).&\label{eq:modcoalg5}
 \end{eqnarray}
 \end{defi}

 \begin{defi}\label{de:bicomalg} Let $(A, \mu, \D, \a_{A}, \b_{A}, \psi_{A}, \om_{A})$ be a $\l$-infBH-bialgebra, $(M, \rho, \vp, \psi_M, \om_M)$ be a BiHom-bicomodule over $(A, \D_A, \psi_A, \om_A)$ and $(M, \mu_M, \a_M, \b_M)$ be a BiHom-associative algebra. The 8-tuple $(M, \mu_M, \rho, \vp, \a_M, \b_M, \psi_M, \om_M)$ is called a {\bf BiHom-comodule algebra} if any two of $\a_M, \b_M, \psi_M, \om_M$ commute and the following conditions hold:
 \begin{eqnarray}
 &\psi_M(m n)=\psi_M(m)\psi_M(n),~~\om_M(m n)=\om_M(m)\om_M(n),&\label{eq:comodalg1}\\
 &(\a_A\o \a_M)\circ \rho=\rho\circ \a_M,~~(\b_A\o \b_M)\circ \rho=\rho\circ \b_M,&\label{eq:comodalg2}\\
 &(\a_M\o \a_A)\circ \vp=\vp\circ \a_M,~~(\b_M\o \b_A)\circ \vp=\vp\circ \b_M,&\label{eq:comodalg3}\\
 &(m n)_{(-1)}\o(m n)_{(0)}=\a_{A}(m_{(-1)})\o m_{(0)}\psi_{M}(n),&\label{eq:comodalg4}\\
 &(m n)_{(0)}\o(m n)_{(1)}=\om_{M}(m)n_{(0)}\o\b_{A}(n_{(1)}).&\label{eq:comodalg5}
 \end{eqnarray}
 \end{defi}
 
 \begin{thm}\label{thm:infbipro} Let $(A,\mu,\D,\a_{A},\b_{A},\psi_{A},\om_{A})$ be a $\l$-infBH-bialgebra, $(M, \mu_M, \tl, \tr, \a_M, \b_M)$ be a BiHom-bimodule algebra over $(A, \mu_A, \a_A, \b_A)$ and at the same time, $(M, \D_M, \rho, \vp, \psi_M, \om_M)$ be a BiHom-bicomodule coalgebra over $(A, \D_A, \psi_A, \om_A)$. Then $(A\oplus M, \mu_{A\oplus M}, \D_{A\oplus M}, \a_{A}\oplus\a_{M}, \b_{A}\oplus\b_{M}, \psi_{C}\oplus\psi_{M}, \om_{C}\oplus\om_{M})$ consisting of infBH-product algebra $(A\# M, \a_{A}\oplus\a_{M}, \b_{A}\oplus\b_{M})$ and infBH-coproduct coalgebra $(A\times M, \psi_{A}\oplus\psi_{M}, \om_{A}\oplus\om_{M})$ is a $\l$-infBH-bialgebra if and only if
 \begin{enumerate}
   \item $(M,\tl,\tr,\rho,\varphi,\a_{M},\b_{M},\psi_{M},\om_{M})$ is a $\l$-infBH-Hopf bimodule over  $(A, \mu, \D, \a_{A}, \b_{A}, \psi_{A}, \om_{A})$;
   \item $(M, \D_M, \tl, \tr, \a_M, \b_M, \psi_M, \om_M)$ is a BiHom-bimodule coalgebra;
   \item $(M, \mu_M, \rho, \vp, \a_M, \b_M, \psi_M, \om_M)$ is a BiHom-bicomodule algebra;
   \item
   \begin{eqnarray}\label{eq:yd}
   &&\D_M(m n)=\om_{M}(m)n_{1}\o\b_{M}(n_{2})+\a_{M}(m_{1})\o m_{2}\psi_{M}(n)+\l\a_{M}\om_{M}(m)\o \b_{M}\psi_{M}(n)\nonumber\\
   &&\hspace{25mm}+\om_{M}(m)\tr n_{(-1)}\o\b_{M}(n_{(0)})+\a_{M}(m_{(0)})\o m_{(1)}\tl\psi_{M}(n).
   \end{eqnarray}
 \end{enumerate}
 In this case, we call this bialgebra a {\bf $\l$-infBH-biproduct bialgebra} and  refer to it  by $(A^\#_\times M, \a_{A}\oplus\a_{M}, \b_{A}\oplus\b_{M}, \psi_{A}\oplus\psi_{M}, \om_{A}\oplus\om_{M})$.
 \end{thm}

 \begin{proof} By Corollary \ref{cor:bimodalg}, $(A\# M, \a_{A}\oplus\a_{M}, \b_{A}\oplus\b_{M})$ is an infBH-product algebra since $(M, \mu_M, \tl, \tr$, $\a_M, \b_M)$ is a BiHom-bimodule algebra over $(A, \mu_A, \a_A, \b_A)$. By  Corollary \ref{cor:bicomodcoalg}, $(A\times M, \psi_{A}\oplus\psi_{M}, \om_{A}\oplus\om_{M})$ is an infBH-coproduct coalgebra since $(M, \D_M, \rho, \vp, \psi_M, \om_M)$ is a BiHom-bicomodule coalgebra over $(A, \D_A, \psi_A, \om_A)$. It is easy to prove by hypothesis that $\a_{A}\oplus\a_{M}$, $\b_{A}\oplus\b_{M}$ are algebra maps and $\psi_{A}\oplus\psi_{M}$, $\om_{A}\oplus\om_{M}$ are coalgebra maps if and only if Eqs.(\ref{eq:modcoalg1})-(\ref{eq:modcoalg3}) and Eqs.(\ref{eq:comodalg1})-(\ref{eq:comodalg3}) hold. For the rest of the proof, we only need to compare the following two formulas:\\  For all $a, b\in A$ and $m, n\in M$,
 \begin{eqnarray*}
 \D_{A\oplus M}((a+m)\cdot_{A\oplus M}(b+n))\hspace{-3mm}
 &=&\hspace{-3mm}\D_{A\oplus M}(a b+a\tl n+m\tr b+m n)\\
 &=&\hspace{-3mm}(a b)_{1}\o (a b)_{2}+(a\tl n)_{(-1)}\o(a\tl n)_{(0)}
 +(m\tr b)_{(-1)}\o(m\tr b)_{(0)}\\
 &&\hspace{-3mm}+(m n)_{(-1)}\o(m n)_{(0)}+(a\tl n)_{(0)}\o (a\tl n)_{(1)}+(m\tr b)_{(0)}\o (m\tr b)_{(1)}\\
 &&\hspace{-3mm}+(m n)_{(0)}\o(m n)_{(1)}+(a\tl n)_{1}\o(a\tl n)_{2}
 +(m\tr b)_{1}\o(m\tr b)_{2}\\
 &&\hspace{-3mm}+(m n)_{1}\o(m n)_{2}
 \end{eqnarray*}
 and (write $\D_{A\oplus M}(a+m)=(a+m)_{[1]}\o (a+m)_{[2]}$)
 \begin{eqnarray*}
 &&(\om_{A}\oplus\om_{M})(a+m)\cdot_{A\oplus M}(b+n)_{[1]}\o(\b_{A}\oplus\b_{M})(b+n)_{[2]}+(\a_{A}\oplus\a_{M})(a+m)_{[1]}\\
 &&\qquad\o (a+m)_{[2]}\cdot_{A\oplus M}(\psi_{A}\oplus\psi_{M})(b+n)+\l(\a_{A}\om_{A}(a)+\a_{M}\om_{M}(m))\o (\b_{A}\psi_{A}(b)+\b_{M}\psi_{M}(n))\\
 &&=\om_{A}(a)b_{1}\o\b_{A}(b_{2})+\om_{M}(m)\tr b_{1}\o\b_{A}(b_{2})+\om_{A}(a)n_{(-1)}\o\b_{M}(n_{(0)})\\
 &&\hspace{5mm}+\om_{M}(m)\tr n_{(-1)}\o\b_{M}(n_{(0)})+\om_{A}(a)\tl n_{(0)}\o\b_{A}(n_{(1)})+\om_{M}(m)n_{(0)}\o\b_{A}(n_{(1)})\\
 &&\hspace{5mm}+\om_{A}(a)\tl n_{1}\o\b_{M}(n_{2})+\om_{M}(m)n_{1}\o\b_{M}(n_{2})+\a_{A}(a_{1})\o a_{2}\psi_{A}(b)+\a_{A}(a_{1})\o a_{2}\tl\psi_{M}(n)\\
 &&\hspace{5mm}+\a_{A}(m_{(-1)})\o m_{(0)}\tr\psi_{A}(b)+\a_{A}(m_{(-1)})\o m_{(0)}\psi_{M}(n)+\a_{M}(m_{(0)})\o m_{(1)}\psi_{A}(b)\\
 &&\hspace{5mm}+\a_{M}(m_{(0)})\o m_{(1)}\tl\psi_{M}(n)+\a_{M}(m_{1})\o m_{2}\tr\psi_{A}(b)+\a_{M}(m_{1})\o m_{2}\psi_{M}(n)\\
 &&\hspace{5mm}+\l\a_{A}\om_{A}(a)\o \b_{A}\psi_{A}(b)
 +\l\a_{A}\om_{A}(a)\o \b_{M}\psi_{M}(n)+\l\a_{M}\om_{M}(m)\o \b_{A}\psi_{A}(b)\\
 &&\hspace{5mm}+\l\a_{M}\om_{M}(m)\o \b_{M}\psi_{M}(n).
 \end{eqnarray*}
 These finish the proof.
 \end{proof}

 \begin{rmk} Similar to Majid-Radford's bosonization, $(M, \mu_M, \D_M, \a_M, \b_M, \psi_M, \om_M)$ is not an infBH-bialgebra of weight $\l$ in the usual sense since the compatibility  condition Eq.(\ref{eq:yd}) is deformed by $\om_{M}(m)\tr n_{(-1)}\o\b_{M}(n_{(0)})+\a_{M}(m_{(0)})\o m_{(1)}\tl\psi_{M}(n)$.
 \end{rmk}

\subsection{$\l$-Rota-Baxter BiHom-bimodules} We extend in the following the notion of Rota-Baxter bimodule \cite{BGM} to the BiHom-type.

 \begin{defi}(\cite[Definition 2.1]{MYZZ}) Let $\l$ be an element in $K$. A 4-tuple $(A,R,\a,\b)$ is called a {\bf $\l$-Rota-Baxter BiHom-algebra} if $(A, \mu, \a,\b)$ is a BiHom-associative algebra and $R: A \rightarrow A$ is $K$-linear map satisfying the following conditions:
 \begin{eqnarray}
 &\a\circ R=R\circ \a,\quad \b\circ R=R\circ \b,&\mlabel{eq:CO4.67}\\
 &R(a)R(b)=R(aR(b))+R(R(a)b)+ \l R(ab).&\mlabel{eq:CO4.20}
 \end{eqnarray}
 The map $R$ is called a {\bf Rota-Baxter operator of weight $\l$ on $(A, \mu, \a, \b)$}.
 \end{defi}
 
 \begin{defi}\mlabel{de:20.06} Let $(A, \mu,\a_{A}, \b_{A}, R_{A})$ be a $\l$-Rota-Baxter BiHom-associative algebra. A {\bf $\l$-Rota-Baxter BiHom-bimodule on $(A,\mu,\a_{A},\b_{A}, R_{A})$} is a 6-tuple $(M, \tl, \tr, \a_{M}, \b_{M}, R_{M})$, where $(M, \tl, \tr, \a_{M}$, $\b_{M})$ is an $(A,\mu,\a_{A},\b_{A})$-bimodule and $R_{M}: M\longrightarrow M$ is a linear map such that $\a_{M}\circ R_{M}=R_{M}\circ\a_{M},\ \b_{M}\circ R_{M}=R_{M}\circ\b_{M}$, and
 \begin{eqnarray}
 &R_{A}(a)\tl R_{M}(m)=R_{M}(a\tl R_{M}(m)+R_{A}(a)\tl m+ \l a\tl m),&\mlabel{eq:20.23}\\
 &R_{M}(m)\tr R_{A}(a)=R_{M}(m\tr R_{A}(a)+R_{M}(m)\tr a+ \l m\tr a),&\mlabel{eq:20.24}
 \end{eqnarray}
 for all $a\in A$ and $m\in M$.
 \end{defi}

 \begin{rmk}
 A $\l$-Rota-Baxter BiHom-associative algebra $(A, \mu, \a, \b, R)$ is a $\l$-Rota-Baxter BiHom-bimodule over itself.
 \end{rmk}

 \begin{pro} \mlabel{pro:rbbimod} Let $(A, \mu,\a_{A}, \b_{A}, R_{A})$ be a $\l$-Rota-Baxter BiHom-associative algebra, $M$ be a vector space, $\a_M, \b_M, R_M: M\lr M$ be  three commuting linear maps and $\tl: A\o M\lr M$ (write $\tl(a\o m)=a\tl m$), $\tr: M\o A\lr M$ (write $\tr(m\o a)=m\tr a$) be two linear maps. Define a multiplication on $A\oplus M$ by
 \begin{eqnarray}
 (a+m)\cdot (a'+m'):=a a'+a\tl m'+m\tr a', \quad \forall a, a'\in A \hbox{~and~} m, m'\in M.
 \end{eqnarray}
 Then $(A\oplus M, \cdot, \a_A\oplus \a_M, \b_A\oplus \b_M, R_A\oplus R_M)$ is a $\l$-Rota-Baxter BiHom-associative algebra if and only if $(M, \tl, \tr, \a_{M}, \b_{M}, R_{M})$ is a $\l$-Rota-Baxter BiHom-bimodule on $(A,\mu,\a_{A},\b_{A}, R_{A})$, where $(\a_A\oplus \a_M)(a+m):=\a_A(a)+\a_M(m)$, similarly for $\b_A\oplus \b_M$ and $R_A\oplus R_M$. In this case we call this algebra {\bf semi-direct product} of $(A, \mu,\a_{A}, \b_{A}, R_{A})$ and $(M, \tl, \tr,\a_{M}, \b_{M}, R_{M})$.
 \end{pro}

 \begin{proof} {\bf Step 1.} Let the product in $M$ is zero in Lemma \ref{lem:smashproduct}, then we obtain that $(A\oplus M,\cdot,\a_{A}\oplus\a_{M},\b_{A}\oplus\b_{M})$ is a BiHom-associative algebra if and only if $(M,\tl,\tr,\a_{M},\b_{M})$ is an $(A,\mu,\a_{A},\b_{A})$-bimodule.

 {\bf Step 2.} We check that $R_{A}\oplus R_{M}$ is a Rota-Baxter operator of weight $\l$ on $(A\oplus M,\cdot,\a_{A}\oplus\a_{M},\b_{A}\oplus\b_{M})$ if and only if Eqs.(\mref{eq:20.23}) and (\mref{eq:20.24}) hold. For all $a, a'\in A$ and $m, m'\in M$, one has
 \begin{eqnarray*}
 (R_{A}\oplus R_{M})(a+m)\cdot(R_{A}\oplus R_{M})(a'+m')=R_{A}(a)R_{A}(a')+R_{A}(a)\tl R_{M}(m')+R_{M}(m)\tr R_{A}(a')
 \end{eqnarray*}
 and
 \begin{eqnarray*}
 &&\hspace{-10mm}(R_{A}\oplus R_{M})((a+m)\cdot(R_{A}\oplus R_{M})(a'+m'))+(R_{A}\oplus R_{M})((R_{A}\oplus R_{M})(a+m)\cdot(a'+m'))\\
 &&+\l(R_{A}\oplus R_{M})((a+m)\cdot(a'+m'))\\
 &=&R_{A}(aR_{A}(a'))+R_{A}(R_{A}(a)a')+\l R_{A}(aa')
 +R_{M}(a\tl R_{M}(m'))+R_{M}(R_{A}(a)\tl m')+\l R_{M}(a\tl m')\\
 &&+R_{M}(m\tr R_{A}(a'))+ R_{M}(R_{M}(m)\tr a')+\l R_{M}(m\tr a').
 \end{eqnarray*}
 Therefore $R_{A}\oplus R_{M}$ is a Rota-Baxter operator of weight $\l$ on $(A\oplus M,\cdot,\a_{A}\oplus\a_{M},\b_{A}\oplus\b_{M})$ if and only if
\begin{eqnarray*}
 &&R_{A}(a)\tl R_{M}(m')+R_{M}(m)\tr R_{A}(a')=R_{M}(a\tl R_{M}(m'))+R_{M}(R_{A}(a)\tl m')\\
 &&+\l R_{M}(a\tl m')+R_{M}(m\tr R_{A}(a'))+ R_{M}(R_{M}(m)\tr a')+\l R_{M}(m\tr a')
 \end{eqnarray*}
 if and only if Eqs.(\mref{eq:20.23}) and (\mref{eq:20.24}) hold. These finish the proof.
 \end{proof}

 \begin{rmk} (1) In \cite[Proposition 4 or Theorem 5]{HHS}, the authors only proved that if $(M, \tl, \tr, \a_{M}$, $\b_{M})$ is a BiHom-bimodule on $(A,\mu,\a_{A},\b_{A})$, then $(A\oplus M, \cdot, \a_{A}\oplus\a_{M}, \b_{A}\oplus\b_{M})$ is a BiHom-associative algebra. Here we show that the converse also holds.

 (2) When the structure maps $\a=\b=\id$ in Proposition \mref{pro:rbbimod}, we obtain \cite[Proposition 2.10]{BGM}.
 \end{rmk}
 
 We recall from \cite{MM} the definition of (anti)quasitriangular unitary $\l$-infBH-bialgebra.

 For convenience, we follow the notations in \cite{LMMP3} or \cite{MY}. Let $(A,\mu,\a,\b)$ be a unitary BiHom-associative algebra, $\psi,\om: A\lr A$ be linear maps and  $r\in A\o A$. We define the following elements in $A\o A\o A$:
 \begin{eqnarray*}
 &r_{12}r_{23}=\a(r^{1})\o r^{2}\bar{r}^{1}\o\b(\bar{r}^{2}),~~ r_{13}r_{12}=\om(r^{1})\bar{r}^{1}\o \b(\bar{r}^{2})\o\a\psi(r^{2}),&\\
 &r_{23}r_{13}=\b\om(r^{1})\o\a(\bar{r}^{1})\o \bar{r}^{2}\psi(r^{2}),\qquad r_{13}=\om(r^{1})\o 1\o\psi(r^{2}).& 
 \end{eqnarray*}

 \begin{defi}\mlabel{de:waybe} Let $(A,\mu,1,\a,\b)$ be a unitary BiHom-associative algebra, $\psi,\om: A\lr A$ be linear maps and $r\in A\o A$. We call
 \begin{eqnarray}
 r_{13}r_{12}-r_{12}r_{23}+r_{23}r_{13}=\l r_{13} \mlabel{eq:waybe}
 \end{eqnarray}
 a {\bf $\l$-associative BiHom-Yang-Baxter equation (abbr. $\l$-abhYBe) in $(A,\mu,1,\a,\b)$} where $\l$ is a given element in $K$.
 \end{defi}
 
 \begin{defi}\mlabel{de:qt} Let $(A,\mu,1,\alpha,\beta)$ be a unitary BiHom-associative algebra such that $\alpha,\beta$ are bijective, $\psi,\omega: A\longrightarrow A$ be linear maps, $r=r^{1}\otimes r^{2}\in A\otimes A$ be $\alpha,\beta,\psi,\omega$-invariant and moreover Eqs.(\mref{eq:12.1}),  (\mref{eq:12.3}) and (\mref{eq:12.30}) hold.
 A {\bf (resp. anti)quasitriangular unitary $\l$-infBH-bialgebra} is a 8-tuple $(A, \mu, 1, \a, \b, \psi, \om, r)$ consisting of a unitary BiHom-associative algebra $(A, \mu, 1, \a, \b)$ and a solution $r\in A\o A$ of a $(resp. -\l)\l$-abhYBe. In this case the comultiplication is defined by
 \begin{eqnarray}\mlabel{eq:14.5}
 \D_{r}(a)=\om\a^{-1}(a) r^{1}\o\b(r^{2})-\a(r^{1})\o r^{2} \psi\b^{-1}(a)-\l (\om(a)\o 1)
 \end{eqnarray}
 (resp. \begin{eqnarray}
 &\widetilde{\D}_{r}(a)=\om\a^{-1}(a) r^{1}\o\b(r^{2})-\a(r^{1})\o r^{2} \psi\b^{-1}(a)-\l (1\o \psi(a)).&\mlabel{eq:14.25}
 \end{eqnarray}).
 \end{defi}
 
 \begin{pro}\mlabel{pro:14.8} \cite{MM} With the assumption of Definition \mref{de:qt}, $(A, \mu, \D=\D_r, 1, \alpha, \beta, \psi, \omega)$, where $\D_r$ is defined by Eq.(\mref{eq:14.5}), is a quasitriangular unitary $\l$-infBH-bialgebra if and only if
 \begin{eqnarray}
 &(\Delta\otimes\psi)(r)=-r_{23}r_{13}&\mlabel{eq:14.8}
 \end{eqnarray}
 or
 \begin{eqnarray}
 &(\omega\otimes \Delta)(r)=r_{13}r_{12}-\lambda(r_{13}+r_{12}).&\mlabel{eq:14.9}
 \end{eqnarray}
 holds.
 \end{pro}

 Rota-Baxter BiHom-bimodules can be constructed by  modules over (anti)quasitriangular unitary $\l$-infBH-bialgebras.

 \begin{pro}\mlabel{pro:20.08} Let $(A,\mu,1,\a_{A},\b_{A},\psi_{A},\om_{A},r)$ be a quasitriangular unitary $\l$-infBH-bialgebra and $(M,\tl,\tr,\a_{M},\b_{M})$ be an $(A,\mu,\a_{A},\b_{A})$-bimodule with invertible maps $\a_{M},\b_{M}$.
 Then $(M, \tl, \tr, \a_{M}$, $\b_{M}, R_{M})$ is a $\l$-Rota-Baxter BiHom-bimodule on $(A,\mu,\a_{A},\b_{A}, R_{A})$, where the maps $R_A$ and $R_M$ are defined by
 \begin{eqnarray*}
 &R_A(a)=-\b_A^{2}\psi_A(r^{1})(\a_A^{-1}\b_A^{-1}(a)\a_A\om_A(r^{2}))&
 \end{eqnarray*}
 and
 \begin{eqnarray*}
 &R_{M}(m)=-\b_{A}^{2}\psi_{A}(r^{1})\tl
 (\a_{M}^{-1}\b_{M}^{-1}(m)\tr\a_{A}\om_{A}(r^{2})),&\mlabel{eq:4.81}
 \end{eqnarray*}
 respectively.
 \end{pro}

 \begin{proof} By \cite[Corollary 2.6]{MLi}, we know that $(A, \mu, \a_{A}, \b_{A}, R_{A})$ is a $\l$-Rota-Baxter BiHom-algebra. For all $a\in A$ and $m\in M$, we have
 \begin{eqnarray*}
 R_{A}(a)\tl R_{M}(m)
 \hspace{-5mm}&\stackrel{(\mref{eq:1.15})}=&\hspace{-5mm}\a_{A}\b_{A}^{2}\psi(r^{1})\tl
 ((\a_{A}^{-1}\b_{A}^{-1}(a)
 (\om_{A}(r^{2})\psi_{A}(\bar{r}^{1})))
 \tl(\a_{M}^{-1}\b_{M}^{-1}(m)\tr\a_{A}\om_{A}(\bar{r}^{2})))\\
 \hspace{-5mm}&\stackrel{(\mref{eq:waybe})}=&\hspace{-5mm}
 \b_{A}^{2}\psi_{A}(r^{1})\b_{A}^{2}\psi_{A}\om_{A}^{-1}(\bar{r}^{1})\tl
 ((\a_{A}^{-1}\b_{A}^{-1}(a)\b_{A}(\bar{r}^{2}))
 \tl(\a_{M}^{-1}\b_{M}^{-1}(m)\tr\a_{A}^{2}\b_{A}^{-1}\om_{A}(r^{2})))\\
 \hspace{-5mm}&&\hspace{-5mm}+\b_{A}^{3}\psi_{A}(\bar{r}^{1})\tl(\a_{A}^{-1}\b_{A}^{-1}(a)
 \a_{A}(r^{1})\tl(\a_{M}^{-1}\b_{M}^{-1}(m)\tr
 \a_{A}\b_{A}^{-1}\om_{A}\psi_{A}^{-1}(r^{2})\a_{A}\b_{A}^{-1}\om_{A}(\bar{r}^{2})))\\
 \hspace{-5mm}&&\hspace{-5mm}-\l\b_{A}^{2}\psi_{A}(r^{1})\tl(\b_{A}^{-1}(a)\tl
 (\a_{M}^{-1}\b_{M}^{-1}(m)\tr\a_{A}\b_{A}^{-1}\om_{A}(r^{2})))\\
 \hspace{-5mm}&\stackrel{(\mref{eq:1.15})(\mref{eq:1.16})}=&\hspace{-3mm}\b_{A}^{2}\psi_{A}(\bar{r}^{1})\tl
 (\b_{A}\psi_{A}(r^{1})(\a_{A}^{-1}\b_{A}^{-2}(a)\a_{A}\b_{A}^{-1}\om_{A}(r^{2}))
 \tl(\a_{M}^{-1}\b_{M}^{-1}(m)\tr
 \a_{A}\b_{A}^{-1}\om_{A}(\bar{r}^{2})))\\
 \hspace{-5mm}&&\hspace{-5mm}+\b_{A}^{2}\psi_{A}(\bar{r}^{1})\tl
 ((\a_{A}^{-2}\b_{A}^{-1}(a)\a_{A}^{-1}\b_{A}\psi_{A}(r^{1})\tl (\a_{M}^{-2}\b_{M}^{-1}(m)\tr\om_{A}(r^{2})))\tr\a_{A}\om_{A}(\bar{r}^{2}))\\
 \hspace{-5mm}&&\hspace{-5mm}-\l\b_{A}^{2}\psi_{A}(r^{1})\tl(\b_{A}^{-1}(a)\tl
 (\a_{M}^{-1}\b_{M}^{-1}(m)\tr\a_{A}\b_{A}^{-1}\om_{A}(r^{2})))\\
 \hspace{-5mm}&=&\hspace{-5mm}R_{M}(R_{A}(a)\tl m)+R_{M}(a\tl R_{M}(m))+\l R_{M}(a\tl m).
 \end{eqnarray*}
 So Eq.(\mref{eq:20.23}) holds. Eq.(\mref{eq:20.24}) can be checked similarly. These finish the proof.
 \end{proof}

 \begin{rmk} We can also obtain a similar result  to Proposition \mref{pro:20.08} for an anti-quasitriangular unitary $\l$-infBH-bialgebra $(A, \mu, 1, \a_{A}, \b_{A}, \psi_{A}, \om_{A}, r)$. In this case, $R_{M}$ is defined by
 \begin{eqnarray*}
 &R_{M}(m)=\b_{A}^{2}\psi_{A}(r^{1})\tl
 (\a_{M}^{-1}\b_{M}^{-1}(m)\tr\a_{A}\om_{A}(r^{2})).&\mlabel{eq:04.081}
 \end{eqnarray*}
 \end{rmk}

\subsection{BiHom-dendriform bimodules from Rota-Baxter BiHom-bimodules} Firstly we recall the notion of BiHom-dendriform bimodule. 
 \begin{defi}\mlabel{de:020.04} (\cite[Definition 3.1]{LMMP1}) A {\bf BiHom-dendriform algebra} is a 5-tuple $(A, \prec, \succ, \a, \b)$ consisting of a linear space $A$ and linear maps $\prec,\succ:A\o A\longrightarrow A$ and $\a,\b:A\longrightarrow A$ satisfying, for all $a,b,c\in A$, the following conditions:
 \begin{eqnarray}
 &\a\circ\b=\b\circ\a,&\mlabel{eq:020.07}\\
 &\a(a\prec b)=\a(a)\prec\a(b),\quad \a(a\succ b)=\a(a)\succ\a(b),&\mlabel{eq:020.08}\\
 &\b(a\prec b)=\b(a)\prec\b(b),\quad \b(a\succ b)=\b(a)\succ\b(b),&\mlabel{eq:020.09}\\
 &(a\prec b)\prec\b(c)=\a(a)\prec(b\prec c)+\a(a)\prec(b\succ c),\quad&\mlabel{eq:020.11}\\
 &\a(a)\succ(b\prec c)=(a\succ b)\prec\b(c),\quad&\mlabel{eq:020.12}\\
 &\a(a)\succ(b\succ c)=(a\prec b)\succ\b(c)+(a\succ b)\succ\b(c).\quad&\mlabel{eq:020.13}
 \end{eqnarray}
 \end{defi}

 \begin{defi}\mlabel{de:20.04}(\cite[Definition 3.10]{Lar}) Let $(A,\prec,\succ,\a_{A},\b_{A})$ be a BiHom-dendriform algebra. A {\bf BiHom-dendriform bimodule} over $(A, \prec, \succ, \a_{A}, \b_{A})$ is a vector space $M$ together with linear maps
 \begin{eqnarray*}
 && A\o M \longrightarrow M,\quad\  A \o M \longrightarrow M,\quad\  M\o A \longrightarrow M,\quad\  M\o A \longrightarrow M \\
 && a\o m \longmapsto a\succ m,\ a\o m \longmapsto a\prec m,\ m\o a \longmapsto m\succ a,\ m\o a \longmapsto m\prec a
 \end{eqnarray*}
 and two commuting linear maps $\a_{M}, \b_{M}: M\longrightarrow M$, such that
 \begin{eqnarray}
 &\a_{M}(a\prec m)=\a_{A}(a)\prec\a_{M}(m),\quad \a_{M}(m\prec a)=\a_{M}(m)\prec\a_{A}(a),&\mlabel{eq:20.07}\\
 &\a_{M}(a\succ m)=\a_{A}(a)\succ\a_{M}(m),\quad \a_{M}(m\succ a)=\a_{M}(m)\succ\a_{A}(a),&\mlabel{eq:20.08}\\
 &\b_{M}(a\prec m)=\b_{A}(a)\prec\b_{M}(m),\quad \b_{M}(m\prec a)=\b_{M}(m)\prec\b_{A}(a),&\mlabel{eq:20.09}\\
 &\b_{M}(a\succ m)=\b_{A}(a)\succ\b_{M}(m),\quad \b_{M}(m\succ a)=\b_{M}(m)\succ\b_{A}(a),&\mlabel{eq:20.10}\\
 &(a\prec b)\prec\b_{M}(m)=\a_{A}(a)\prec(b\prec m)+\a_{A}(a)\prec(b\succ m),\quad&\mlabel{eq:20.11}\\
 &\a_{A}(a)\succ(b\prec m)=(a\succ b)\prec\b_{M}(m),\quad&\mlabel{eq:20.12}\\
 &\a_{A}(a)\succ(b\succ m)=(a\prec b)\succ\b_{M}(m)+(a\succ b)\succ\b_{M}(m),\quad&\mlabel{eq:20.13}\\
 &(a\prec m)\prec\b_{A}(b)=\a_{A}(a)\prec(m\prec b)+\a_{A}(a)\prec(m\succ b),\quad&\mlabel{eq:20.14}\\
 &\a_{A}(a)\succ(m\prec b)=(a\succ m)\prec\b_{A}(b),\quad&\mlabel{eq:20.15}\\
 &\a_{A}(a)\succ(m\succ b)=(a\prec m)\succ\b_{A}(b)+(a\succ m)\succ\b_{A}(b),\quad&\mlabel{eq:20.16}\\
 &(m\prec a)\prec\b_{A}(b)=\a_{M}(m)\prec(a\prec b)+\a_{M}(m)\prec(a\succ b),\quad&\mlabel{eq:20.17}\\
 &\a_{M}(m)\succ(a\prec b)=(m\succ a)\prec\b_{A}(b),\quad&\mlabel{eq:20.18}\\
 &\a_{M}(m)\succ(a\succ b)=(m\prec a)\succ\b_{A}(b)+(m\succ a)\succ\b_{A}(b),\quad&\mlabel{eq:20.19}
 \end{eqnarray}
 for all $a,b \in A$ and $m\in M$.
 \end{defi} 
 
 \begin{pro}\mlabel{pro:D29.07} Let $(A,\prec,\succ, \a_{A}, \b_{A})$ be a BiHom-dendriform algebra, $M$ a vector space together with linear maps
 \begin{eqnarray*}
 && A\o M \longrightarrow M,\quad\  A \o M \longrightarrow M,\quad\  M\o A \longrightarrow M,\quad\  M\o A \longrightarrow M \\
 && a\o m \longmapsto a\succ m,\ a\o m \longmapsto a\prec m,\ m\o a \longmapsto m\succ a,\ m\o a \longmapsto m\prec a
 \end{eqnarray*}
 and two commuting linear maps $\a_{M}, \b_{M}: M\longrightarrow M$. Define two linear maps $\prec_{A\oplus M}, \succ_{A\oplus M}$ on $A\oplus M$ by
 \begin{eqnarray}
 &(a+m)\prec_{A\oplus M}(b+n)=a\prec b+a\prec n+m\prec b,&\mlabel{eq:D029.26}\\
 &(a+m)\succ_{A\oplus M}(b+n)=a\succ b+a\succ n+m\succ b,&\mlabel{eq:D029.27}
 \end{eqnarray}
 for all $a, b\in A$ and $m, n\in M$. Then $(A\oplus M, \prec_{A\oplus M}, \succ_{A\oplus M}, \a_A\oplus\a_M, \b_A\oplus\b_M)$, where $(\a_{A}\oplus\a_{M})(a+m)=\a_{A}(a)+\a_{M}(m)$, $(\b_{A}\oplus\b_{M})(a+m)=\b_{A}(a)+\b_{M}(m)$, is a BiHom-dendriform algebra if and if $(M, \prec, \succ, \a_{M}, \b_{M})$ is a BiHom-dendriform bimodule over $(A,\prec,\succ, \a_{A}, \b_{A})$.
 \end{pro}

 \begin{proof} Straightforward.
 \end{proof}

 \begin{rmk} \cite[Proposition 3.11]{Lar} contains only the sufficiency in Proposition \mref{pro:D29.07}.
 \end{rmk} 

 \begin{pro}\mlabel{pro:20.07} Let $(A, \mu,\a_{A}, \b_{A}, R_{A})$ be a $\l$-Rota-Baxter BiHom-associative algebra and $(M, \tl$, $\tr, \a_{M}, \b_{M}, R_{M})$ be a $\l$-Rota-Baxter BiHom-bimodule on $(A,\mu,\a_{A},\b_{A}, R_{A})$. Define new actions of $A$ on $M$ by
 \begin{eqnarray*}
 &&a\succ m=R_{A}(a)\tl m,\
 a\prec m=a\tl R_{M}(m)+\l a\tl m,\\
 &&m\succ a=R_{M}(m)\tr a,\ m\prec a=m\tr R_{A}(a)+\l m\tr a
 \end{eqnarray*}
 (resp.
 \begin{eqnarray*}
 &&a\succ m=R_{A}(a)\tl m+\l a\tl m,\
 a\prec m=a\tl R_{M}(m),\\
 &&m\succ a=R_{M}(m)\tr a+\l m\tr a,\ m\prec a=m\tr R_{A}(a)
 \end{eqnarray*})
 for all $a\in A$ and $m\in M$. Then $(M, \prec, \succ, \a_{M}, \b_{M})$ is a BiHom-dendriform bimodule over the BiHom-dendriform algebra $(A, \prec, \succ, \a_{A}, \b_{A})$, where the linear maps $\prec,\succ: A\o A\lr A$ are defined by
 \begin{eqnarray*}
 &a\prec b=aR(b)+\l ab,\quad a\succ b=R(a)b&
 \end{eqnarray*}
 (resp.
 \begin{eqnarray*}
 &a\prec b=aR(b),\quad a\succ b=R(a)b+\l ab&
 \end{eqnarray*})
 respectively.
 \end{pro}

 \begin{proof} By \cite[Theorem 2.5]{MYZZ}, we know that $(A, \prec, \succ, \a_{A}, \b_{A})$ is a BiHom-dendriform algebra. Next we only check Eqs.(\mref{eq:20.11})-(\mref{eq:20.13}) of the former case as follows, and others are similar.
 \begin{eqnarray*}
 (a\prec b)\prec\b_{M}(m)
 \hspace{-3mm}&=&\hspace{-3mm}(aR_{A}(b))\tl R_{M}(\b_{M}(m))+\l (ab)\tl R_{M}(\b_{M}(m))
 +\l (aR_{A}(b))\tl\b_{M}(m)\\
 &&\hspace{-3mm}+\l^{2}(ab)\tl \b_{M}(m) \\
 \hspace{-3mm}&\stackrel{(\mref{eq:1.15})}=&\hspace{-4mm}\a_{A}(a)\tl(R_{A}(b)\tl R_{M}(m))+\l (ab)\tl R_{M}(\b_{M}(m))
 +\l (aR_{A}(b))\tl\b_{M}(m)\\
 &&\hspace{-3mm}+ \l^{2}(ab)\tl \b_{M}(m)\\
 \hspace{-3mm}&\stackrel{(\mref{eq:20.23})}=&\hspace{-3mm}\a_{A}(a)\tl R_{M}(b\tl R_{M}(m))+\a_{A}(a)\tl R_{M}(R_{A}(b)\tl m)\\
 \hspace{-3mm}&&\hspace{-3mm}+\l\a_{A}(a)\tl R_{M}(b\tl m)+\l (ab)\tl R_{M}(\b_{M}(m))
 +\l (aR_{A}(b))\tl\b_{M}(m)\\
&&\hspace{-3mm}+ \l^{2}(ab)\tl \b_{M}(m)\\
 \hspace{-3mm}&=&\hspace{-3mm}\a_{A}(a)\prec(b\prec m)+\a_{A}(a)\prec(b\succ m),
 \end{eqnarray*}
 \begin{eqnarray*}
 \a_{A}(a)\succ(b\prec m)
 &=& R_{A}(\a_{A}(a))\tl(b\tl R_{M}(m))+\l R_{A}(\a_{A}(a))\tl(b\tl m) \\
 &\stackrel{(\mref{eq:1.15})}=&(R_{A}(a)b)\tl R_{M}(\b_{M}(m))+\l (R_{A}(a)b)\tl \b_{M}(m)=(a\succ b)\prec\b_{M}(m),
 \end{eqnarray*}
 and
 \begin{eqnarray*}
 \a_{A}(a)\succ(b\succ m)
 &=&R_{A}(\a_{A}(a))\tl(R_{A}(b)\tl m) \\
 &\stackrel{(\mref{eq:1.15})}=&(R_{A}(a)R_{A}(b))\tl \b_{M}(m)\\
 &=&R_{A}(aR_{A}(b))\tl \b_{M}(m)+R_{A}(R_{A}(a)b)\tl \b_{M}(m)+\l R_{A}(ab)\tl \b_{M}(m)\\
 &=&(a\prec b)\succ\b_{M}(m)+(a\succ b)\succ\b_{M}(m),
 \end{eqnarray*}
 finishing the proof.
 \end{proof}

 \begin{cor}\mlabel{cor:20.09} Let $(A,\mu,1,\a_{A},\b_{A},\psi_{A},\om_{A},r)$ be a quasitriangular unitary $\l$-infBH-bialgebra, and $(M, \tl, \tr, \a_{M}, \b_{M})$ be an arbitrary $(A, \mu, \a_{A}, \b_{A})$-bimodule such that $\a_{M}, \b_{M}$ are invertible. Define new actions of $A$ on $M$ by
 \begin{eqnarray*}
 &&a\succ m=-\b_{A}^{2}\psi_{A}(r^{1})(\a_{A}^{-1}\b_{A}^{-1}(a)\a_{A}\om_{A}(r^{2}))\tl m,\\
 &&a\prec m=-a\tl(\b_{A}^{2}\psi_{A}(r^{1})\tl
 (\a_{M}^{-1}\b_{M}^{-1}(m)\tr\a_{A}\om_{A}(r^{2})))+\l a\tl m,\\
 &&m\succ a=-(\b_{A}^{2}\psi_{A}(r^{1})\tl (\a_{M}^{-1}\b_{M}^{-1}(m)\tr\a_{A}\om_{A}(r^{2})))\tr a,\\
 &&m\prec a=-m\tr \b_{A}^{2}\psi_{A}(r^{1})
 (\a_{A}^{-1}\b_{A}^{-1}(a)\a_{A}\om_{A}(r^{2}))+\l m\tr a,
 \end{eqnarray*}
 (resp.
 \begin{eqnarray*}
 &&a\succ m=-\b_{A}^{2}\psi_{A}(r^{1})(\a_{A}^{-1}\b_{A}^{-1}(a)\a_{A}\om_{A}(r^{2}))\tl m
 +\l a\tl m,\\
 &&a\prec m=-a\tl(\b_{A}^{2}\psi_{A}(r^{1})\tl
 (\a_{M}^{-1}\b_{M}^{-1}(m)\tr\a_{A}\om_{A}(r^{2}))),\\
 &&m\succ a=-(\b_{A}^{2}\psi_{A}(r^{1})\tl (\a_{M}^{-1}\b_{M}^{-1}(m)\tr\a_{A}\om_{A}(r^{2})))\tr a
 +\l m\tr a,\\
 &&m\prec a=-m\tr\b_{A}^{2}\psi_{A}(r^{1})
 (\a_{A}^{-1}\b_{A}^{-1}(a)\a_{A}\om_{A}(r^{2})).
 \end{eqnarray*}).
 Then the 5-tuple $(M,\prec,\succ,\a_{M},\b_{M})$ is a BiHom-dendriform bimodule over the BiHom-dendriform algebra $(A,\prec,\succ,\a_A,\b_A)$, where
 \begin{eqnarray*}
 &a\succ b=-(\b_A^{2}\psi_A(r^{1})(\a_A^{-1}\b_A^{-1}(a)\a_A\om_A(r^{2})))b,\ a\prec b=-a(\b_A^{2}\psi_A(r^{1})(\a_A^{-1}\b_A^{-1}(b)\a_A\om_A(r^{2})))+\l a b&
 \end{eqnarray*}
 (resp.
 \begin{eqnarray*}
 &a\succ b=-(\b_A^{2}\psi_A(r^{1})(\a_A^{-1}\b_A^{-1}(a)\a_A\om_A(r^{2})))b+\l a b,\ a\prec b=-a(\b_A^{2}\psi_A(r^{1})(\a_A^{-1}\b_A^{-1}(b)\a_A\om_A(r^{2}))).&
 \end{eqnarray*})
 \end{cor}

 \begin{proof} It can be proved by Propositions \mref{pro:20.08} and \mref{pro:20.07}.
 \end{proof} 

\subsection{BiHom-pre-Lie bimodules from BiHom-dendriform bimodules} 
 \begin{defi}\mlabel{de:13.1} (\cite[Definition 3.1]{LMMP6}) A {\bf (left) BiHom-pre-Lie algebra} is a 4-tuple $(A, \circ, \a, \b)$, where $A$ is a vector space together with two commuting linear maps $\a, \b: A \lr  A$ and a linear map $\circ: A \o A \lr  A$ such that $\a(a\circ b)=\a(a)\circ\a(b)$, $\b(a\circ b)=\b(a)\circ\b(b)$ and
 \begin{eqnarray}
 &\a\b(a)\circ(\a(b)\circ c)-(\b(a)\circ\a(b))\circ\b(c)=\a\b(b)\circ(\a(a)\circ c)
 -(\b(b)\circ\a(a))\circ\b(c),&\mlabel{eq:13.1}
 \end{eqnarray}
 for all $a,b,c \in A$.
 \end{defi} 
 
 \begin{defi}\mlabel{de:20.03}(\cite[Definition 3.3]{Lar}) Let $(A,\circ,\a_{A},\b_{A})$ be a (left) BiHom-pre-Lie algebra. A {\bf BiHom-pre-Lie bimodule} over $(A,\circ,\a_{A},\b_{A})$ is a vector space $M$ endowed with maps
 \begin{eqnarray*}
 &A \o M \longrightarrow M,\hspace{3mm}a\o m\longmapsto a\circ_{\ell} m,\hspace{8mm} M\o A \longrightarrow M, \hspace{3mm} m\o a\longmapsto m\circ_r a&
 \end{eqnarray*}
 and two commuting linear maps $\a_{M}, \b_{M}: M\longrightarrow M$, such that
 \begin{eqnarray}
 &&\a_{M}(a\circ_{\ell} m)=\a_{A}(a)\circ_{\ell}\a_{M}(m),\quad \a_{M}(m\circ_r a)=\a_{M}(m)\circ_r\a_{A}(a),\mlabel{eq:20.03}\\
 &&\b_{M}(a\circ_{\ell} m)=\b_{A}(a)\circ_{\ell}\b_{M}(m),\quad \b_{M}(m\circ_r a)=\b_{M}(m)\circ_r\b_{A}(a),\mlabel{eq:20.04}\\
 &&\a_{A}\b_{A}(a)\circ_{\ell}(\a_{A}(b)\circ_{\ell} m)-(\b_{A}(a)\circ \a_{A}(b))\circ_{\ell}\b_{M}(m)\mlabel{eq:20.05}\\
 &&\hspace{40mm}=\a_{A}\b_{A}(b)\circ_{\ell}(\a_{A}(a)\circ_{\ell} m)
 -(\b_{A}(b)\circ \a_{A}(a))\circ_{\ell}\b_{M}(m),\nonumber\\
 &&\a_{A}\b_{A}(a)\circ_{\ell}(\a_{M}(m)\circ_r c)-(\b_{A}(a)\circ_{\ell}\a_{M}(m))\circ_r\b_{A}(c)\mlabel{eq:20.06}\\
 &&\hspace{40mm}=\a_{M}\b_{M}(m)\circ_r(\a_{A}(a)\circ c)
 -(\b_{M}(m)\circ_r\a_{A}(a))\circ_r\b_{A}(c),\nonumber
 \end{eqnarray}
 for all $a,b,c \in A$ and $m\in M$.
 \end{defi} 

 \begin{rmk}
 A BiHom-pre-Lie algebra $(A,\circ,\a_{A},\b_{A})$ is a BiHom-pre-Lie bimodule over itself.
 \end{rmk}

 \begin{pro}\mlabel{pro:D29.06} Let $(A, \circ, \a_{A}, \b_{A})$ be a BiHom-pre-Lie algebra, $M$ a vector space endowed with maps
 \begin{eqnarray*}
 &A \o M \longrightarrow M,\hspace{3mm}a\o m\longmapsto a\circ_{\ell} m,\hspace{8mm} M\o A \longrightarrow M, \hspace{3mm} m\o a\longmapsto m\circ_r a&
 \end{eqnarray*}
 and two commuting linear maps $\a_{M}, \b_{M}: M\longrightarrow M$. Define a linear map $\circ_{A\oplus M}$ on $A\oplus M$ by
 \begin{eqnarray}
 &(a+m)\circ_{A\oplus M}(a'+m')=a\circ a'+a\circ_{\ell} m'+m\circ_r a',&\mlabel{eq:D029.20}
 \end{eqnarray}
 for all $a, a'\in A$ and $m, m'\in M$. Then $(A\oplus M, \circ_{A\oplus M}, \a_{A}\oplus\a_{M},\b_{A}\oplus\b_{M})$, where $(\a_{A}\oplus\a_{M})(a+m)=\a_{A}(a)+\a_{M}(m)$, $(\b_{A}\oplus\b_{M})(a+m)=\b_{A}(a)+\b_{M}(m)$, is a BiHom-pre-Lie algebra if and only if $(M, \circ_{\ell}, \circ_r, \a_M, \b_M)$ is a BiHom-pre-Lie bimodule over $(A, \circ, \a_{A}, \b_{A})$.
 \end{pro}

 \begin{proof} Straightforward.
 \end{proof} 

 \begin{rmk} \cite[Proposition 3.4]{Lar} contains only the sufficiency in Proposition \mref{pro:D29.06}.
 \end{rmk}

 \begin{pro}\mlabel{pro:20.05} Let $(M, \prec, \succ, \a_{M}, \b_{M})$ be a BiHom-dendriform bimodule over the BiHom-dendriform algebra $(A, \prec, \succ, \a_{A}, \b_{A})$. Then $(M, \circ_\ell, \circ_r, \a_{M}, \b_{M})$ is a BiHom-pre-Lie bimodule over the BiHom-pre-Lie algebra $(A, \circ, \a_A, \b_A)$, where
 \begin{eqnarray}
 &a\circ_\ell m=a\succ m-\a_{M}^{-1}\b_{M}(m)\prec\a_{A}\b_{A}^{-1}(a),&\mlabel{eq:20.20}\\
 &m\circ_r a=m\succ a-\a_{A}^{-1}\b_{A}(a)\prec\a_{M}\b_{M}^{-1}(m),& \mlabel{eq:20.21}\\
 &a\circ b=a\succ b-\a_{A}^{-1}\b_{A}(b)\prec\a_{A}\b_{A}^{-1}(a),&
 \end{eqnarray}
 for all $a, b\in A$ and $m\in M$.
 \end{pro}

 \begin{proof} By \cite[Proposition 3.6]{LMMP6}, we know that $(A, \circ, \a_A, \b_A)$ is a BiHom-pre-Lie algebra. For $a,b,c \in A$ and $m\in M$, we have
 \begin{eqnarray*}
 &&\hspace{-30mm}\a_{A}\b_{A}(a)\circ_\ell(\a_{A}(b) \circ_\ell m)-(\b_{A}(a)\circ \a_{A}(b))\circ_\ell \b_{M}(m) \\
 &=&\hspace{-10mm}\a_{A}\b_{A}(a)\succ(\a_{A}(b)\succ m)-\a_{A}\b_{A}(a)\succ(\a_{M}^{-1}\b_{M}(m)\prec \a_{A}^{2}\b_{A}^{-1}(b))\\
 &&\hspace{-10mm}-(\b_{A}(b)\succ\a_{M}^{-1}\b_{M}(m))\prec\a_{A}^{2}(a)+(\a_{M}^{-2}\b_{M}^{2}(m) \prec \a_{A}(b))\prec\a_{A}^{2}(a)\\
 &&\hspace{-10mm}-(\b_{A}(a)\succ\a_{A}(b))\succ\b_{M}(m)+(\b_{A}(b) \prec\a_{A}(a))\succ\b_{M}(m)\\
 &&\hspace{-10mm}+\a_{M}^{-1}\b_{M}^{2}(m)\prec(\a_{A}(a) \succ \a_{A}^{2}\b_{A}^{-1}(b))-\a_{M}^{-1}\b_{M}^{2}(m)\prec
 (\a_{A}(b)\prec\a_{A}^{2}\b_{A}^{-1}(a))\\
 &\stackrel{(\mref{eq:20.13})(\mref{eq:20.15})(\mref{eq:20.17})}=&
 (\b_{A}(a)\prec\a_{A}(b))\succ\b_{M}(m)+(\b_{A}(a)\succ\a_{A}(b))\succ\b_{M}(m)\\
 &&\hspace{-10mm}-(\b_{A}(a)\succ\a_{M}^{-1}\b_{M}(m)) \prec\a_{A}^{2}(b)- (\b_{A}(b)\succ \a_{M}^{-1}\b_{M}(m))\prec\a_{A}^{2}(a)\\
 &&\hspace{-10mm}+\a_{M}^{-1}\b_{M}^{2}(m)\prec(\a_{A}(b) \prec \a_{A}^{2}\b_{A}^{-1}(a))+\a_{M}^{-1}\b_{M}^{2}(m)\prec(\a_{A}(b) \succ \a_{A}^{2}\b_{A}^{-1}(a))\\
 &&\hspace{-10mm}-(\b_{A}(a)\succ \a_{A}(b))\succ\b_{M}(m)+(\b_{A}(b)\prec\a_{A}(a))\succ\b_{M}(m)\\
 &&\hspace{-10mm}+\a_{M}^{-1}\b_{M}^{2}(m) \prec (\a_{A}(a)\succ\a_{A}^{2}\b_{A}^{-1}(b))
 -\a_{M}^{-1}\b_{M}^{2}(m)\prec(\a_{A}(b)\prec\a_{A}^{2}\b_{A}^{-1}(a))\\
 &=&\hspace{-10mm}(\b_{A}(a)\prec\a_{A}(b))\succ\b_{M}(m)-(\b_{A}(a)\succ\a_{M}^{-1}\b_{M}(m)) \prec\a_{A}^{2}(b)\\
 &&\hspace{-10mm}- (\b_{A}(b)\succ \a_{M}^{-1}\b_{M}(m))\prec\a_{A}^{2}(a)
 +\a_{M}^{-1}\b_{M}^{2}(m)\prec(\a_{A}(b) \succ \a_{A}^{2}\b_{A}^{-1}(a))\\
 &&\hspace{-10mm}+(\b_{A}(b)\prec\a_{A}(a))\succ\b_{M}(m)
 +\a_{M}^{-1}\b_{M}^{2}(m) \prec (\a_{A}(a)\succ\a_{A}^{2}\b_{A}^{-1}(b))
 \end{eqnarray*}
 Observe that above formula is symmetric in $a$ and $b$, hence Eq.(\mref{eq:20.05}) holds. Moreover,
 \begin{eqnarray*}
 &&\hspace{-30mm}\a_{A}\b_{A}(a)\circ_\ell (\a_{M}(m)\circ_r c)-(\b_{A}(a) \circ_\ell \a_{M}(m))\circ_r\b_{A}(c)\\
 &=&\hspace{-10mm}\a_{A}\b_{A}(a)\succ(\a_{M}(m)\succ c)-\a_{A}\b_{A}(a)\succ(\a_{A}^{-1}\b_{A}(c)\prec \a_{M}^{2}\b_{M}^{-1}(m))\\
 &&\hspace{-10mm}-(\b_{M}(m)\succ\a_{A}^{-1}\b_{A}(c))\prec\a_{A}^{2}(a)+(\a_{A}^{-2}\b_{A}^{2}(c) \prec \a_{M}(m))\prec\a_{A}^{2}(a)\\
 &&\hspace{-10mm}-(\b_{A}(a)\succ\a_{M}(m))\succ\b_{A}(c)+(\b_{M}(m) \prec\a_{A}(a))\succ\b_{A}(c)\\
 &&\hspace{-10mm}+\a_{A}^{-1}\b_{A}^{2}(c)\prec(\a_{A}(a) \succ \a_{M}^{2}\b_{M}^{-1}(m))-\a_{A}^{-1}\b_{A}^{2}(c)\prec
 (\a_{M}(m)\prec\a_{A}^{2}\b_{A}^{-1}(a))\\
 &\stackrel{(\mref{eq:20.12})(\mref{eq:20.14})(\mref{eq:20.16})}=&
 (\b_{A}(a)\prec\a_{M}(m))\succ\b_{A}(c)-(\b_{A}(a)\succ\a_{A}^{-1}\b_{A}(c)) \prec\a_{M}^{2}(m)\\
 &&\hspace{-10mm}- (\b_{M}(m)\succ \a_{A}^{-1}\b_{A}(c))\prec\a_{A}^{2}(a)
 +\a_{A}^{-1}\b_{A}^{2}(c)\prec(\a_{M}(m) \succ \a_{A}^{2}\b_{A}^{-1}(a))\\
 &&\hspace{-10mm}+(\b_{M}(m)\prec\a_{A}(a))\succ\b_{A}(c)
 +\a_{A}^{-1}\b_{A}^{2}(c) \prec (\a_{A}(a)\succ\a_{M}^{2}\b_{M}^{-1}(m)),
 \end{eqnarray*}
 on the other hand,
 \begin{eqnarray*}
 &&\hspace{-30mm}\a_{M}\b_{M}(m)\circ_r (\a_{A}(a)\circ c)-(\b_{M}(m) \circ_r \a_{A}(a))\circ_r \b_{A}(c)\\
 &=&\hspace{-10mm}\a_{M}\b_{M}(m)\succ(\a_{A}(a)\succ c)-\a_{M}\b_{M}(m)\succ(\a_{A}^{-1}\b_{A}(c)\prec \a_{A}^{2}\b_{A}^{-1}(a))\\
 &&\hspace{-10mm}-(\b_{A}(a)\succ\a_{A}^{-1}\b_{A}(c))\prec\a_{M}^{2}(m)+(\a_{A}^{-2}\b_{A}^{2}(c) \prec \a_{A}(a))\prec\a_{M}^{2}(m)\\
 &&\hspace{-10mm}-(\b_{M}(m)\succ\a_{A}(a))\succ\b_{A}(c)+(\b_{A}(a) \prec\a_{M}(m))\succ\b_{A}(c)\\
 &&\hspace{-10mm}+\a_{A}^{-1}\b_{A}^{2}(c)\prec(\a_{M}(m) \succ \a_{A}^{2}\b_{A}^{-1}(a))-\a_{A}^{-1}\b_{A}^{2}(c)\prec
 (\a_{A}(a)\prec\a_{M}^{2}\b_{M}^{-1}(m))\\
 &\stackrel{(\mref{eq:20.11})(\mref{eq:20.18})(\mref{eq:20.19})}=& (\b_{M}(m)\prec\a_{A}(a))\succ\b_{A}(c)
 -(\b_{M}(m)\succ\a_{A}^{-1}\b_{A}(c)) \prec\a_{A}^{2}(a)\\
 &&\hspace{-10mm}- (\b_{A}(a)\succ \a_{A}^{-1}\b_{A}(c))\prec\a_{M}^{2}(m)
 +\a_{A}^{-1}\b_{A}^{2}(c)\prec(\a_{A}(a) \succ \a_{M}^{2}\b_{M}^{-1}(m))\\
 &&\hspace{-10mm}+(\b_{A}(a)\prec\a_{M}(m))\succ\b_{A}(c)
 +\a_{A}^{-1}\b_{A}^{2}(c) \prec (\a_{M}(m)\succ\a_{A}^{2}\b_{A}^{-1}(a)),
 \end{eqnarray*}
 as desired.
 \end{proof} 
 
 \begin{pro}\mlabel{pro:20.12} Let $(A,\mu,\D,\a_{A},\b_{A},\psi_{A},\om_{A})$ be a $\l$-infBH-bialgebra and $(M,\gamma,\rho,\nu,\varphi,\a_{M},\b_{M}$, $\psi_{M},\om_{M})$ be  a $\l$-infBH-Hopf bimodule over $(A, \mu, \D, \a_{A}, \b_{A}, \psi_{A}, \om_{A})$. Then $(M, \star, \a_{M}^{2}\b_{M}, \a_{M}^{2}\b_{M}^{2}\psi_{M}\om_{M})$ is a BiHom-pre-Lie bimodule over the BiHom-pre-Lie algebra $(A, \star, \a_{A}^{2}\b_{A}, \a_{A}^{2}\b_{A}^{2}\psi_{A}\om_{A})$, where
 \begin{eqnarray*}
 &a\star m=(\b_{A}^{2}\psi_{A}(m_{-1})\a_{A}(a))\tl \a_{M}^{2}\b_{M}\om_{M}(m_{0})
 +(\b_{M}^{2}\psi_{M}(m_{(0)})\tr \a_{A}(a))\tr \a_{A}^{2}\b_{A}\om_{A}(m_{(1)}),&\\
 &m\star a=(\b_{A}^{2}\psi_{A}(a_{1})\tl \a_{M}(m))\tr \a_{A}^{2}\b_{A}\om_{A}(a_{2}),&\\
 &a\star b=(\b^{2}\psi(b_{1})\a(a))\a^{2}\b\om(b_{2}), \quad \forall a, b\in A, m\in M.&\mlabel{eq:13.6}
 \end{eqnarray*}
 \end{pro}

 \begin{proof} Firstly by \cite[Theorem 3.27]{MM}, $(A, \star, \a_{A}^{2}\b_{A}, \a_{A}^{2}\b_{A}^{2}\psi_{A}\om_{A})$ is a BiHom-pre-Lie algebra. In what follows, we only do the calculations below, and the rest are similar. For all $a, b\in A$ and $m\in M$, we have
 \begin{eqnarray*}
 \rho(m\star b)&\stackrel{(\mref{eq:20.02})}=&\a_{M}(\b_{A}^{2}\psi_{A}(b_{1})
 \tl\a_{M}(m))_{-1}\o (\b_{A}^{2}\psi_{A}(b_{1})
 \tl\a_{M}(m))_{0}\tr\a_{A}^{2}\b_{A}\psi_{A}\om_{A}(b_{2})\\
 &\stackrel{(\mref{eq:12.13})}=&\a_{A}\b_{A}^{2}\psi_{A}\om_{A}(b_{1})
 \a_{A}^{2}(m_{-1})\o
 \a_{M}\b_{M}(m_{0})\tr\a_{A}^{2}\b_{A}\psi_{A}\om_{A}(b_{2})\\
 &&+\a_{A}^{2}\b_{A}^{2}\psi_{A}(b_{11})\o(\b_{A}^{2}\psi_{A}(b_{12})\tl
 \a_{M}\psi_{M}(m))
 \tr\a_{A}^{2}\b_{A}\psi_{A}\om_{A}(b_{2})\\
 &&+\l\a_{A}^{2}\b_{A}^{2}\psi_{A}\om_{A}(b_{1})\o\a_{M}\b_{M}\psi_{M}(m)
 \tr\a_{A}^{2}\b_{A}\psi_{A}\om_{A}(b_{2}),
 \end{eqnarray*}
 and
 \begin{eqnarray*}
 \varphi(m\star b)&\stackrel{(\mref{eq:12.13a})}=&
 (\b_{A}^{2}\psi_{A}\om_{A}(b_{1})\tl\a_{M}\om_{M}(m))\tr
 \a_{A}^{2}\b_{A}\om_{A}(b_{21})\o\a_{A}^{2}\b_{A}^{2}\om_{A}(b_{22})\\
 &&+\a_{M}(\b_{A}^{2}\psi_{A}(b_{1})\tl\a_{M}(m))_{(0)}\o
 (\b_{A}^{2}\psi_{A}(b_{1})\tl\a_{M}(m))_{(1)}
 \a_{A}^{2}\b_{A}\psi_{A}\om_{A}(b_{2})\\
 &&+\l\a_{A}\b_{A}^{2}\psi_{A}\om_{A}(b_{1})\tl
 \a_{M}^{2}\om_{M}(m)
 \o\a_{A}^{2}\b_{A}^{2}\psi_{A}\om_{A}(b_{2})\\
 &\stackrel{(\mref{eq:20.01})}=&
 (\b_{A}^{2}\psi_{A}(b_{11})\tl\a_{M}\om_{M}(m))\tr
 \a_{A}^{2}\b_{A}\om_{A}(b_{12})\o\a_{A}^{2}\b_{A}^{2}\psi_{A}\om_{A}(b_{2})\\
 &&+\a_{A}\b_{A}^{2}\psi_{A}\om_{A}(b_{1})\tl\a_{M}^{2}(m_{(0)})\o
 \a_{A}\b_{A}(m_{(1)})\a_{A}^{2}\b_{A}\psi_{A}\om_{A}(b_{2})\\
 &&+\l\a_{A}\b_{A}^{2}\psi_{A}\om_{A}(b_{1})\tl\a_{M}^{2}\om_{M}(m)
 \o\a_{A}^{2}\b_{A}^{2}\psi_{A}\om_{A}(b_{2}).
 \end{eqnarray*}
 Thus,
 \begin{eqnarray*}
 &&\hspace{-20mm}\a_{A}^{4}\b_{A}^{3}\psi_{A}\om_{A}(a)\star(\a_{M}^{2}\b_{M}(m)\star b)-(\a_{A}^{2}\b_{A}^{2}\psi_{A}\om_{A}(a)\star\a_{M}^{2}\b_{M}(m))\star \a_{A}^{2}\b_{A}^{2}\psi_{A}\om_{A}(b)\\
 &=&\hspace{-5mm}
 (\a_{A}\b_{A}^{4}\psi_{A}^{2}\om_{A}(b_{1})\a_{A}^{4}\b_{A}^{3}\psi_{A}(m_{-1}))
 \a_{A}^{5}\b_{A}^{3}\psi_{A}\om_{A}(a)\tl
 (\a_{M}^{5}\b_{M}^{3}\om_{M}(m_{0})\tr
 \a_{A}^{4}\b_{A}^{2}\psi_{A}\om_{A}^{2}(b_{2}))\\
 &&\hspace{-5mm}+\a_{A}^{2}\b_{A}^{4}\psi_{A}^{2}(b_{11})
 \a_{A}^{5}\b_{A}^{3}\psi_{A}\om_{A}(a)
 \tl((\a_{A}^{2}\b_{A}^{3}\psi_{A}\om_{A}(b_{12})\tl
 \a_{M}^{5}\b_{M}^{2}\psi_{M}\om_{M}(m))
 \tr\a_{A}^{4}\b_{A}^{2}\psi_{A}\om_{A}^{2}(b_{2}))\\
 &&\hspace{-5mm}+\l\a_{A}^{2}\b_{A}^{4}\psi_{A}^{2}\om_{A}(b_{1})
 \a_{A}^{5}\b_{A}^{3}\psi_{A}\om_{A}(a)\tl
 (\a_{M}^{5}\b_{M}^{3}\psi_{M}\om_{M}(m)
 \tr\a_{A}^{4}\b_{A}^{2}\psi_{A}\om_{A}^{2}(b_{2}))\\
 &&\hspace{-5mm}+(((\b_{A}^{4}\psi_{A}^{2}(b_{11})\tl
 \a_{M}^{3}\b_{M}^{3}\psi_{M}\om_{M}(m))\tr
 \a_{A}^{2}\b_{A}^{3}\psi_{A}\om_{A}(b_{12}))\tr
 \a_{A}^{5}\b_{A}^{3}\psi_{A}\om_{A}(a))\tr
 \a_{A}^{4}\b_{A}^{3}\psi_{A}\om_{A}^{2}(b_{2})\\
 &&\hspace{-5mm}+((\a_{A}\b_{A}^{4}\psi_{A}^{2}\om_{A}(b_{1})\tl
 \a_{M}^{4}\b_{M}^{3}\psi_{M}(m_{(0)}))\tr
 \a_{A}^{5}\b_{A}^{3}\psi_{A}\om_{A}(a))\tr
 \a_{A}^{5}\b_{A}^{3}\om_{A}(m_{(1)})
 \a_{A}^{4}\b_{A}^{2}\psi_{A}\om_{A}^{2}(b_{2})\\
 &&\hspace{-5mm}+\l((\a_{A}\b_{A}^{4}\psi_{A}^{2}\om_{A}(b_{1})\tl
 \a_{M}^{4}\b_{M}^{3}\psi_{M}\om_{M}(m))
 \tr
 \a_{A}^{5}\b_{A}^{3}\psi_{A}\om_{A}(a))\tr
 \a_{A}^{4}\b_{A}^{3}\psi_{A}\om_{A}^{2}(b_{2})\\
 &&\hspace{-5mm}-(\a_{A}^{2}\b_{A}^{4}\psi_{A}^{2}\om_{A}(b_{1})\tl
 (\a_{A}^{3}\b_{A}^{3}\psi_{A}(m_{-1})
 \a_{A}^{4}\b_{A}^{2}\psi_{A}\om_{A}(a)\tl
 \a_{M}^{5}\b_{M}^{2}\om_{M}(m_{0})))\tr
 \a_{A}^{4}\b_{A}^{3}\psi_{A}\om_{A}^{2}(b_{2})\\
 &&\hspace{-5mm}-(\a_{A}^{2}\b_{A}^{4}\psi_{A}^{2}\om_{A}(b_{1})\tl
 (\a_{M}^{4}\b_{M}^{3}\psi_{M}(m_{(0)})\tr
 \a_{A}^{4}\b_{A}^{2}\psi_{A}\om_{A}(a)
 \a_{A}^{5}\b_{A}\om_{A}(m_{(1)})))\tr
 \a_{A}^{4}\b_{A}^{3}\psi_{A}\om_{A}^{2}(b_{2})\\
 &\stackrel{(\mref{eq:1.15})(\mref{eq:1.16})}=&\hspace{-1mm}
 (\a_{A}\b_{A}^{4}\psi_{A}^{2}(b_{11})
 (\a_{A}^{3}\b_{A}^{3}\psi_{A}\om_{A}(a)
 \a_{A}^{2}\b_{A}^{2}\psi_{A}\om_{A}(b_{12}))\tl
 \a_{M}^{5}\b_{M}^{3}\psi_{M}\om_{M}(m))
 \tr\a_{A}^{4}\b_{A}^{3}\psi_{A}\om_{A}^{2}(b_{2})\\
 &&\hspace{-5mm}+\l(\a_{A}\b_{A}^{4}\psi_{A}^{2}\om_{A}(b_{1})
 \a_{A}^{4}\b_{A}^{3}\psi_{A}\om_{A}(a)\tl
 \a_{M}^{5}\b_{M}^{3}\psi_{M}\om_{M}(m))
 \tr\a_{A}^{4}\b_{A}^{3}\psi_{A}\om_{A}^{2}(b_{2})\\
 &&\hspace{-5mm}+(\a_{A}^{2}\b_{A}^{4}\psi_{A}^{2}(b_{11})\tl
 \a_{M}^{5}\b_{M}^{3}\psi_{M}\om_{M}(m))\tr
 \a_{A}^{3}\b_{A}^{3}\psi_{A}\om_{A}(b_{12})
 (\a_{A}^{5}\b_{A}^{2}\psi_{A}\om_{A}(a)
 \a_{A}^{4}\b_{A}\psi_{A}\om_{A}^{2}(b_{2}))\\
 &&\hspace{-5mm}+\l(\a_{A}^{2}\b_{A}^{4}\psi_{A}^{2}\om_{A}(b_{1})\tl
 \a_{M}^{5}\b_{M}^{3}\psi_{M}\om_{M}(m))
 \tr
 \a_{A}^{5}\b_{A}^{3}\psi_{A}\om_{A}(a)\a_{A}^{4}\b_{A}^{2}\psi_{A}\om_{A}^{2}(b_{2})\\
 &\stackrel{\triangle}=&\hspace{-5mm} I.
 \end{eqnarray*}
 On the other hand, by Eqs.(\mref{eq:12.4}) and (\mref{eq:1.9}), one can get
 \begin{eqnarray*}
 \D(a\star b)&=&\a_{A}\b_{A}^{2}\psi_{A}(b_{11})
 (\a_{A}\om_{A}(a)\a_{A}^{2}\om_{A}(b_{12}))\o
 \a_{A}^{2}\b_{A}^{2}\psi_{A}\om_{A}(b_{2})\\
 &&+\a_{A}\b_{A}^{2}\psi_{A}\om_{A}(b_{1})\a_{A}^{2}(a_{1})
 \o\a_{A}\b_{A}(a_{2})\a_{A}^{2}\b_{A}\psi_{A}\om_{A}(b_{2})\\
 &&+\l\a_{A}\b_{A}^{2}\psi_{A}\om_{A}(b_{1})\a_{A}^{2}\om_{A}(a)
 \o\a_{A}^{2}\b_{A}^{2}\psi_{A}\om_{A}(b_{2})\\
 &&+\a_{A}^{2}\b_{A}^{2}\psi_{A}(b_{11})\o\a_{A}\b_{A}^{2}\psi_{A}(b_{12})
 (\a_{A}\psi_{A}(a)\a_{A}^{2}\psi_{A}\om_{A}(b_{2}))\\
 &&+\l\a_{A}^{2}\b_{A}^{2}\psi_{A}\om_{A}(b_{1})
 \o\a_{A}\b_{A}\psi_{A}(a)\a_{A}^{2}\b_{A}\psi_{A}\om_{A}(b_{2}).
 \end{eqnarray*}
 Then we obtain that,
 \begin{eqnarray*}
 &&\hspace{-20mm}\a_{M}^{4}\b_{M}^{3}\psi_{M}\om_{M}(m)\star(\a_{A}^{2}\b_{A}(a)\star b)-(\a_{M}^{2}\b_{M}^{2}\psi_{M}\om_{M}(m)\star\a_{A}^{2}\b_{A}(a))
 \star \a_{A}^{2}\b_{A}^{2}\psi_{A}\om_{A}(b)\\
 &=&\hspace{-8mm}
 (\a_{A}\b_{A}^{4}\psi_{A}^{2}(b_{11})
 (\a_{A}^{3}\b_{A}^{3}\psi_{A}\om_{A}(a)
 \a_{A}^{2}\b_{A}^{2}\psi_{A}\om_{A}(b_{12}))\tl
 \a_{M}^{5}\b_{M}^{3}\psi_{M}\om_{M}(m))\tr
 \a_{A}^{4}\b_{A}^{3}\psi_{A}\om_{A}^{2}(b_{2})\\
 &&\hspace{-8mm}+(\a_{A}\b_{A}^{4}\psi_{A}^{2}\om_{A}(b_{1})\a_{A}^{4}\b_{A}^{3}\psi_{A}(a_{1})
 \tl\a_{M}^{5}\b_{M}^{3}\psi_{M}\om_{M}(m))\tr
 \a_{A}^{5}\b_{A}^{3}\om_{A}(a_{2})
 \a_{A}^{4}\b_{A}^{2}\psi_{A}\om_{A}^{2}(b_{2})\\
 &&\hspace{-8mm}+\l(\a_{A}\b_{A}^{4}\psi_{A}^{2}\om_{A}(b_{1})
 \a_{A}^{4}\b_{A}^{3}\psi_{A}\om_{A}(a)
 \tl\a_{M}^{5}\b_{M}^{3}\psi_{M}\om_{M}(m))\tr
 \a_{A}^{4}\b_{A}^{3}\psi_{A}\om_{A}^{2}(b_{2})\\
 &&\hspace{-8mm}+(\a_{A}^{2}\b_{A}^{4}\psi_{A}^{2}(b_{11})\tl
 \a_{M}^{5}\b_{M}^{3}\psi_{M}\om_{M}(m))
 \tr\a_{A}^{3}\b_{A}^{3}\psi_{A}\om_{A}(b_{12})
 (\a_{A}^{5}\b_{A}^{2}\psi_{A}\om_{A}(a)
 \a_{A}^{4}\b_{A}\psi_{A}\om_{A}^{2}(b_{2}))\\
 &&\hspace{-8mm}+\l(\a_{A}^{2}\b_{A}^{4}\psi_{A}^{2}\om_{A}(b_{1})\tl
 \a_{M}^{5}\b_{M}^{3}\psi_{M}\om_{M}(m))
 \tr\a_{A}^{5}\b_{A}^{3}\psi_{A}\om_{A}(a)
 \a_{A}^{4}\b_{A}^{2}\psi_{A}\om_{A}^{2}(b_{2})\\
 &&\hspace{-8mm}-\a_{A}^{2}\b_{A}^{4}\psi_{A}^{2}\om_{A}(b_{1})\a_{A}^{5}\b_{A}^{3}\psi_{A}(a_{1})
 \tl((\a_{M}^{4}\b_{M}^{3}\psi_{M}\om_{M}(m)\tr
 \a_{A}^{5}\b_{A}^{2}\om_{A}(a_{2}))\tr
 \a_{A}^{4}\b_{A}^{2}\psi_{A}\om_{A}^{2}(b_{2}))\\
 &\stackrel{(\mref{eq:1.15})(\mref{eq:1.16})}=&\hspace{-2mm}I.
 \end{eqnarray*}
 Hence,
 \begin{eqnarray*}
 &&\a_{A}^{4}\b_{A}^{3}\psi_{A}\om_{A}(a)\star(\a_{M}^{2}\b_{M}(m)\star b)-(\a_{A}^{2}\b_{A}^{2}\psi_{A}\om_{A}(a)\star\a_{M}^{2}\b_{M}(m))\star \a_{A}^{2}\b_{A}^{2}\psi_{A}\om_{A}(b)\qquad\qquad\qquad\qquad\\
 &&\qquad\qquad = \a_{M}^{4}\b_{M}^{3}\psi_{M}\om_{M}(m)\star(\a_{A}^{2}\b_{A}(a)\star b)-(\a_{M}^{2}\b_{M}^{2}\psi_{M}\om_{M}(m)\star\a_{A}^{2}\b_{A}(a))\star \a_{A}^{2}\b_{A}^{2}\psi_{A}\om_{A}(b).
 \end{eqnarray*}
 By Eqs.(\mref{eq:12.13}), (\mref{eq:20.02}) and (\mref{eq:01.24}), we have
 \begin{eqnarray*}
 \rho(b\star m)&=&
 (\b_{A}^{2}\psi_{A}\om_{A}(m_{-1})\a_{A}\om_{A}(b))
 \a_{A}^{2}\b_{A}\om_{A}(m_{0-1})\o \a_{M}^{2}\b_{M}^{2}\om_{M}(m_{00})\\
 &&+\a_{A}\b_{A}^{2}\psi_{A}\om_{A}(m_{-1})\a_{A}^{2}(b_{1})\o\a_{A}\b_{A}(b_{2})
 \tl\a_{M}^{2}\b_{M}\psi_{M}\om_{M}(m_{0})\\
 &&+\a_{A}^{2}\b_{A}^{2}\psi_{A}\om_{A}(m_{-1})
 \o\b_{A}^{2}\psi_{A}(m_{0-1})\a_{A}\psi_{A}(b)
 \tl\a_{M}^{2}\b_{M}\om_{M}(m_{00})\\
 &&+\l\a_{A}^{2}\b_{A}^{2}\psi_{A}\om_{A}(m_{-1})\o\a_{A}\b_{A}\psi_{A}(b)
 \tl\a_{M}^{2}\b_{M}\psi_{M}\om_{M}(m_{0})\\
 &&+\l\a_{A}\b_{A}^{2}\psi_{A}\om_{A}(m_{-1})\a_{A}^{2}\om_{A}(b)
 \o\a_{M}^{2}\b_{M}^{2}\psi_{M}\om_{M}(m_{0})\\
 &&+\a_{A}^{2}\b_{A}^{2}\psi_{A}(m_{(0)-1})\o\a_{M}\b_{M}^{2}\psi_{M}(m_{(0)0})
 \tr\a_{A}\psi_{A}(b)\a_{A}^{2}\psi_{A}\om_{A}(m_{(1)}).
 \end{eqnarray*}
The identity 
 \begin{eqnarray*}
 \varphi(b\star m)&=&
 (\b_{A}^{2}\psi_{A}(m_{(0)-1})\a_{A}\om_{A}(b)\tl
 \a_{M}^{2}\b_{M}\om_{M}(m_{(0)0}))
 \o\a_{A}^{2}\b_{A}^{2}\psi_{A}\om_{A}(m_{(1)})\\
 &&+\a_{M}\b_{M}^{2}\psi_{M}\om_{M}(m_{(0)})
 \tr\a_{A}\om_{A}(b)\a_{A}^{2}\om_{A}(m_{(1)1})
 \o\a_{A}^{2}\b_{A}^{2}\om_{A}(m_{(1)2})\\
 &&+\a_{M}\b_{M}^{2}\psi_{M}\om_{M}(m_{(0)})\tr\a_{A}^{2}(b_{1})
 \o\a_{A}\b_{A}(b_{2})\a_{A}^{2}\b_{A}\psi_{A}\om_{A}(m_{(1)})\\
 &&+\l\a_{M}\b_{M}^{2}\psi_{M}\om_{M}(m_{(0)})\tr\a_{A}^{2}\om_{A}(b)
 \o\a_{A}^{2}\b_{A}^{2}\psi_{A}\om_{A}(m_{(1)})\\
 &&+\a_{M}^{2}\b_{M}^{2}\psi_{M}\om_{M}(m_{(0)})\o\a_{A}\b_{A}^{2}\psi_{A}(m_{(1)1})
 (\a_{A}\psi_{A}(b)\a_{A}^{2}\om_{A}(m_{(1)2}))\\
 &&+\l\a_{M}^{2}\b_{M}^{2}\psi_{M}\om_{M}(m_{(0)})
 \o\a_{A}\b_{A}\psi_{A}(b)\a_{A}^{2}\b_{A}\psi_{A}\om_{A}(m_{(1)})
 \end{eqnarray*}
 can be checked by Eqs.(\mref{eq:20.01}), (\mref{eq:12.13a}),  (\mref{eq:01.24aa}) and (\mref{eq:01.24}).

 Therefore,
 \begin{eqnarray*}
 &&\hspace{-20mm}\a_{A}^{4}\b_{A}^{3}\psi_{A}\om_{A}(a)\star(\a_{A}^{2}\b_{A}(b)\star m) -(\a_{A}^{2}\b_{A}^{2}\psi_{A}\om_{A}(a)\star\a_{A}^{2}\b_{A}(b))\star \a_{M}^{2}\b_{M}^{2}\psi_{M}\om_{M}(m)\\
 &\stackrel{(\mref{eq:1.15})(\mref{eq:1.16})}=&
 ((\b_{A}^{4}\psi_{A}^{2}\om_{A}(m_{-1})\a_{A}^{3}\b_{A}^{3}\psi_{A}\om_{A}(b))
 \a_{A}^{2}\b_{A}^{3}\psi_{A}\om_{A}(m_{0-1}))
 \a_{A}^{5}\b_{A}^{3}\psi_{A}\om_{A}(a)
 \tl\a_{M}^{4}\b_{M}^{3}\om_{M}^{2}(m_{00})\\
 &&\hspace{-4mm}+((\b_{A}^{4}\psi_{A}^{2}\om_{A}(m_{-1})\a_{A}^{3}\b_{A}^{3}\psi_{A}\om_{A}(a))
 \a_{A}^{2}\b_{A}^{3}\psi_{A}\om_{A}(m_{0-1}))\a_{A}^{5}\b_{A}^{3}\psi_{A}\om_{A}(b)
 \tl\a_{M}^{4}\b_{M}^{3}\om_{M}^{2}(m_{00})\\
 &&\hspace{-4mm}+\l(\a_{A}\b_{A}^{4}\psi_{A}^{2}\om_{A}(m_{-1})
 \a_{A}^{4}\b_{A}^{3}\psi_{A}\om_{A}(a))
 \a_{A}^{5}\b_{A}^{3}\psi_{A}\om_{A}(b)\tl
 \a_{M}^{4}\b_{M}^{3}\psi_{M}\om_{M}^{2}(m_{0})\\
 &&\hspace{-4mm}+\l(\a_{A}\b_{A}^{4}\psi_{A}^{2}\om_{A}(m_{-1})
 \a_{A}^{4}\b_{A}^{3}\psi_{A}\om_{A}(b))
 \a_{A}^{5}\b_{A}^{3}\psi_{A}\om_{A}(a)\tl
 \a_{M}^{4}\b_{M}^{3}\psi_{M}\om_{M}^{2}(m_{0})\\
 &&\hspace{-4mm}+\a_{A}^{2}\b_{A}^{4}\psi_{A}^{2}(m_{(0)-1})
 \a_{A}^{5}\b_{A}^{3}\psi_{A}\om_{A}(a)\tl
 (\a_{M}^{3}\b_{M}^{3}\psi_{M}\om_{M}(m_{(0)0})
 \tr\a_{A}^{5}\b_{A}^{2}\psi_{A}\om_{A}(b)
 \a_{A}^{4}\b_{A}\psi_{A}\om_{A}^{2}(m_{(1)}))\\
 &&\hspace{-4mm}+\a_{A}^{2}\b_{A}^{4}\psi_{A}^{2}(m_{(0)-1})
 \a_{A}^{5}\b_{A}^{3}\psi_{A}\om_{A}(b)\tl
 (\a_{M}^{3}\b_{M}^{3}\psi_{M}\om_{M}(m_{(0)0})
 \tr\a_{A}^{5}\b_{A}^{2}\psi_{A}\om_{A}(a)
 \a_{A}^{4}\b_{A}\psi_{A}\om_{A}^{2}(m_{(1)}))\\
 &&\hspace{-4mm}+(\a_{M}^{2}\b_{M}^{4}\psi_{M}^{2}\om_{M}(m_{(0)})\tr
 \a_{A}^{5}\b_{A}^{3}\psi_{A}\om_{A}(b))\tr
 \a_{A}^{3}\b_{A}^{3}\psi_{A}\om_{A}(m_{(1)1})
 (\a_{A}^{5}\b_{A}^{2}\psi_{A}\om_{A}(a)
 \a_{A}^{4}\b_{A}\om_{A}^{2}(m_{(1)2}))\\
 &&\hspace{-4mm}+\l(\a_{M}^{2}\b_{M}^{4}\psi_{M}^{2}\om_{M}(m_{(0)})\tr
 \a_{A}^{5}\b_{A}^{3}\psi_{A}\om_{A}(b))
 \tr\a_{A}^{5}\b_{A}^{3}\psi_{A}\om_{A}(a)
 \a_{A}^{4}\b_{A}^{2}\psi_{A}\om_{A}^{2}(m_{(1)})\\
 &&\hspace{-4mm}+(\a_{M}^{2}\b_{M}^{4}\psi_{M}^{2}\om_{M}(m_{(0)})\tr
 \a_{A}^{5}\b_{A}^{3}\psi_{A}\om_{A}(a))\tr
 \a_{A}^{3}\b_{A}^{3}\psi_{A}\om_{A}(m_{(1)1})
 (\a_{A}^{5}\b_{A}^{2}\psi_{A}\om_{A}(b)\a_{A}^{4}\b_{A}\om_{A}^{2}(m_{(1)2}))\\
 &&\hspace{-4mm}+\l(\a_{M}^{2}\b_{M}^{4}\psi_{M}^{2}\om_{M}(m_{(0)})\tr
 \a_{A}^{5}\b_{A}^{3}\psi_{A}\om_{A}(a))
 \tr\a_{A}^{5}\b_{A}^{3}\psi_{A}\om_{A}(b)
 \a_{A}^{4}\b_{A}^{2}\psi_{A}\om_{A}^{2}(m_{(1)})\\
 &\stackrel{\triangle}=&\hspace{-5mm} II.
 \end{eqnarray*}
 Observing the symmetry for $a$ and $b$ in the formula $II$, one has
 \begin{eqnarray*}
 &&\a_{A}^{4}\b_{A}^{3}\psi_{A}\om_{A}(a)\star(\a_{A}^{2}\b_{A}(b)\star m)-(\a_{A}^{2}\b_{A}^{2}\psi_{A}\om_{A}(a)\star\a_{A}^{2}\b_{A}(b))\star \a_{M}^{2}\b_{M}^{2}\psi_{M}\om_{M}(m)\qquad\qquad\qquad\qquad\\
 &&\qquad\qquad = \a_{A}^{4}\b_{A}^{3}\psi_{A}\om_{A}(b)\star(\a_{A}^{2}\b_{A}(a)\star m)-(\a_{A}^{2}\b_{A}^{2}\psi_{A}\om_{A}(b)\star\a_{A}^{2}\b_{A}(a))\star \a_{M}^{2}\b_{M}^{2}\psi_{M}\om_{M}(m),
 \end{eqnarray*}
 completing the proof.
 \end{proof} 

 \begin{pro}\mlabel{pro:020.012} Let $(A,\mu,\D,\a_{A},\b_{A},\psi_{A},\om_{A})$ be a $\l$-infBH-bialgebra such that $\a_{A},\b_{A},\psi_{A},\om_{A}$ are invertible and $(M,\gamma,\rho,\nu,\varphi,\a_{M},\b_{M},\psi_{M},\om_{M})$ be a $\l$-infBH-Hopf bimodule, where $\a_{M},\b_{M},\psi_{M},\om_{M}$ are invertible. Then $(M,\star,\a_{M},\b_{M})$ is a BiHom-pre-Lie bimodule over the BiHom-pre-Lie algebra $(A,\star,\a_{A},\b_{A})$ via
 \begin{eqnarray*}
 &a\star m=(\a_{A}^{-2}\b_{A}\om_{A}^{-1}(m_{-1})\b_{A}^{-1}(a))\tl \psi_{M}^{-1}(m_{0})
 +(\a_{M}^{-2}\b_{M}\om_{M}^{-1}(m_{(0)})\tr \b_{A}^{-1}(a))\tr \psi_{A}^{-1}(m_{(1)})&\\
 &m\star a=(\a_{A}^{-2}\b_{A}\om_{A}^{-1}(a_{1})\tl \b_{M}^{-1}(m))\tr \psi_{A}^{-1}(a_{2}),&\\
 &a\star b=(\a^{-2}\b\om^{-1}(b_{1})\b^{-1}(a))\psi^{-1}(b_{2}).&\mlabel{eq:13.2}
 \end{eqnarray*}
 \end{pro}

 \begin{proof} Firstly by \cite[Theorem 3.26]{MM}, $(A,\star,\a_{A},\b_{A})$ is a BiHom-pre-Lie algebra. The rest is similar to the proof of Proposition \mref{pro:20.12}.
 \end{proof}

\subsection{$\l$-infBH-Hopf bimodules from the (co)modules over (co)quasitriangular (co)unitary $\l$-infBH-bialgebra}
 Left (right) $\l$-infBH-Hopf modules can be constructed from the modules over quasitriangular unitary $\l$-infBH-bialgebra in the following way.
 \begin{thm}\mlabel{thm:12.02}(\cite{MM}) Let $(A, \mu, 1, \a_{A}, \b_{A}, \psi_{A}, \om_{A}, r)$ be a quasitriangular unitary $\l$-infBH-bialgebra and $(M, \gamma, \a_{M}, \b_{M})$ be a left $(A, \mu, \a_{A}, \b_{A})$-module, $\psi_{M}, \om_{M}: M\longrightarrow M$ be linear maps such that $\b_{M}\ci \psi_{M}=\psi_{M}\ci \b_{M}$, $\psi_{M}\ci \g=\g\ci (\psi_{A}\o \psi_{M})$.
 Then $(M, \gamma, \rho, \a_{M}, \b_{M}, \psi_{M}, \om_{M})$ becomes a left $\l$-infBH-Hopf module with the coaction $\rho: M\longrightarrow A\o M$ given by
 \begin{eqnarray*}
 &\rho(m):=-\a_{A}(r^{1})\o r^{2}\tl\psi_{M}\b_{M}^{-1}(m), \forall~ m\in M.&
 \end{eqnarray*}
 \end{thm}
\vspace{-3mm}

 \begin{thm}\mlabel{thm:13.001} Let $(A, \mu, 1, \a_{A}, \b_{A}, \psi_{A}, \om_{A}, r)$ be a quasitriangular unitary $\l$-infBH-bialgebra and $(M, \nu, \a_{M}, \b_{M})$ be a right $(A, \mu, \a_{A}, \b_{A})$-module, $\psi_{M}, \om_{M}: M\longrightarrow M$ be linear maps such that $\a_{M}\ci \om_{M}=\om_{M}\ci \a_{M}$, $\om_{M}\ci \nu=\nu\ci (\om_{M}\o \om_{A})$.
 Then $(M, \nu, \varphi, \a_{M}, \b_{M}, \psi_{M}, \om_{M})$ becomes a right $\l$-infBH-Hopf module with the coaction $\varphi: M\longrightarrow M\o A$ given by
 \begin{eqnarray*}
 &\varphi(m):=\om_{M}\a^{-1}_{M}(m)\tr r^{1}\o\b_{A}(r^{2})-\l\om_{M}(m)\o 1, \forall~ m\in M.&
 \end{eqnarray*}
 \end{thm}

 \begin{proof} We first prove that $(M, \varphi, \psi_{M}, \om_{M})$ is a right $(A, \D_{r}, \psi_{A}, \om_{A})$-comodule as follows. For all $m\in M$, we have
 \begin{eqnarray*}
 (\om_{M}\o \D_r)\varphi(m)
 &\stackrel{(\mref{eq:14.9})}=&\om^{2}_{M}\a^{-1}_{M}(m)\tr
 \b^{-1}_{A}\om_{A}(r^{1})\b^{-1}_{A}(\bar{r}^{1})
 \o\b_{A}(\bar{r}^{2})\o\a_{A}\psi_{A}(r^{2})\\
 &&-\l\om^{2}_{M}\a^{-1}_{M}(m)\tr \b^{-1}_{A}\om_{A}(r^{1})\o 1\o \psi_{A}(r^{2})\\
 &&-\l\om^{2}_{M}\a^{-1}_{M}(m)\tr \b^{-1}_{A}(r^{1})\o r^{2}\o 1+\l^{2}\om^{2}_{M}(m)\o1\o 1\\
 &\stackrel{(\mref{eq:1.15})}=&(\om^{2}_{M}\a^{-2}_{M}(m)\tr\om_{A}\a^{-1}_{A}(r^{1}))\tr\bar{r}^{1}
 \o\b_{A}(\bar{r}^{2})\o\psi_{A}\b_{A}(r^{2})\\
 &&-\l\om^{2}_{M}\a^{-1}_{M}(m)\tr \om_{A}(r^{1})\o 1\o \psi_{A}\b_{A}(r^{2})\\
 &&-\l\om^{2}_{M}\a^{-1}_{M}(m)\tr r^{1} \o \b_{A}(r^{2})\o 1+\l^{2}\om^{2}_{M}(m)\o1\o 1\\
 &\stackrel{(\mref{eq:12.30})}=&(\varphi\o\psi_{A})\varphi(m).
 \end{eqnarray*}
 Next we  check the compatibility condition. For all $a\in A$ and $m\in M$, we have
 \begin{eqnarray*}
 \hspace{-3mm}&&\hspace{-20mm}\om_{M}(m)\tr a_{1}\o \b_{A}(a_{2})+\a_{M}(m_{(0)})\o m_{(1)}\psi_{A}(a)+\l\a_{M}\om_{M}(m)\o \b_{A}\psi_{A}(a)\\
 \hspace{-3mm}&\stackrel{(\mref{eq:14.5})(\mref{eq:1.5})}=&\hspace{-6mm}
 \om_{M}(m)\tr \om_{A}\a^{-1}_{A}(a)r^{1}\o \b^{2}_{A}(r^{2})-\om_{M}(m)\tr\a_{A}(r^{1})\o \b_{A}(r^{2})\psi_{A}(a)\\
 \hspace{-3mm}&&\hspace{-3mm}-\l\om_{M}(m)\tr\om_{A}(a)\o 1+\om_{M}(m)\tr\a_{A}(r^{1})\o\b_{A}(r^{2})\psi_{A}(a)\\
 \hspace{-3mm}&&\hspace{-3mm}-\l\a_{M}\om_{M}(m)\o 1\cdot\psi_{A}(a)
 +\l\a_{M}\om_{M}(m)\o \b_{A}\psi_{A}(a)\\
 \hspace{-3mm}&\stackrel{(\mref{eq:1.5})}=&\hspace{-3mm}\om_{M}(m)\tr \om_{A}\a^{-1}_{A}(a)r^{1}\o \b^{2}_{A}(r^{2})
 -\l\om_{M}(m)\tr\om_{A}(a)\o 1\\
 \hspace{-3mm}&=&\hspace{-3mm}\om_{M}(m)\tr \om_{A}\a^{-1}_{A}(a)\b_A^{-1}(r^{1})\o \b_{A}(r^{2})-\l\om_{M}(m)\tr\om_{A}(a)\o 1\\
 \hspace{-3mm}&\stackrel{(\mref{eq:1.15})(\mref{eq:1.13})}=&\hspace{-3mm}(m\tr a)_{(0)}\o (m\tr a)_{(1)},
 \end{eqnarray*}
 completing the proof.
 \end{proof}
\vspace{-3mm}

 \begin{rmk} The right coaction $\vp$ in Theorem \mref{thm:13.001} is not just a copy of the left coaction $\rho$ in Theorem \mref{thm:12.02}, and the expression of $\vp$ contains the weight $\l$.
 \end{rmk}
\vspace{-3mm}

 Left (right) $\l$-infBH-Hopf modules can be constructed from the modules over anti-quasitriangular unitary $\l$-infBH-bialgebra in the following way.

 \begin{pro}\mlabel{pro:14.25}(\cite{MM}) With the assumptions of Definition \mref{de:qt},  $(A, \mu, \D=\widetilde{\D}_r, 1, \alpha, \beta, \psi, \omega)$, where $\widetilde{\D}_r$ is defined by Eq.(\mref{eq:14.25}), is an anti-quasitriangular unitary $\l$-infBH-bialgebra if and only if
 \begin{eqnarray}
 &(\Delta\otimes\psi)(r)=-r_{23}r_{13}-\lambda(r_{23}+r_{13})&\mlabel{eq:14.28}
 \end{eqnarray}
 or
 \begin{eqnarray}
 &(\omega\otimes \Delta)(r)=r_{13}r_{12}.&\mlabel{eq:14.29}
 \end{eqnarray}
 \end{pro}
\vspace{-3mm}

 \begin{thm}\mlabel{thm:12.02a}(\cite{MM}) Let $(A, \mu, 1, \a_{A}, \b_{A}, \psi_{A}, \om_{A}, r)$ be an anti-quasitriangular unitary $\l$-infBH-bialgebra and $(M, \widetilde{\gamma}, \a_{M}, \b_{M})$ be a left $(A, \mu, \a_{A}, \b_{A})$-module, $\psi_{M}, \om_{M}: M\longrightarrow M$ be linear maps  such that $\b_{M}\ci \psi_{M}=\psi_{M}\ci \b_{M}$, $\psi_{M}\ci \widetilde{\gamma}=\widetilde{\gamma}\ci (\psi_{A}\o \psi_{M})$. Then $(M, \widetilde{\gamma}, \widetilde{\rho}, \a_{M}, \b_{M}, \psi_{M}, \om_{M})$ becomes a left $\l$-infBH-Hopf module with the coaction $\widetilde{\rho}: M\longrightarrow A\o M$ given by
 \begin{eqnarray*}
 &\widetilde{\rho}(m):=-\a_{A}(r^{1})\o r^{2}\tl\psi_{M}\b_{M}^{-1}(m)-\l 1\o \psi_{M}(m), \forall~ m\in M.&
 \end{eqnarray*}
 \end{thm}
\vspace{-3mm}

 \begin{thm}\mlabel{thm:13.002a} Let $(A, \mu, 1, \a_{A}, \b_{A}, \psi_{A}, \om_{A}, r)$ be an anti-quasitriangular unitary $\l$-infBH-bialgebra and $(M, \widetilde{\nu}, \a_{M}, \b_{M})$ be a right $(A, \mu, \a_{A}, \b_{A})$-module, $\psi_{M}, \om_{M}: M\longrightarrow M$ be linear maps such that $\a_{M}\ci \om_{M}=\om_{M}\ci \a_{M}$, $\om_{M}\ci \widetilde{\nu}=\widetilde{\nu}\ci (\om_{M}\o \om_{A})$. Then $(M, \widetilde{\nu}, \widetilde{\varphi}, \a_{M}, \b_{M}, \psi_{M}, \om_{M})$ becomes a right $\l$-infBH-Hopf module with the coaction $\widetilde{\varphi}: M\longrightarrow M\o A$ given by
 \begin{eqnarray*}
 &\widetilde{\varphi}(m):=\om_{M}\a^{-1}_{M}(m)\tr r^{1}\o\b_{A}(r^{2}), \forall~ m\in M.&
 \end{eqnarray*}
 \end{thm}

 \begin{proof} We first prove that $(M, \widetilde{\varphi}, \psi_{M}, \om_{M})$ is a right $(A, \widetilde{\D_r}, \psi_{A}, \om_{A})$-comodule as follows. For all $m\in M$, we have
 \begin{eqnarray*}
 (\om_{M}\o \widetilde{\D_r})\widetilde{\varphi}(m)
 \hspace{-3mm}&\stackrel{(\mref{eq:14.29})}=&\hspace{-3mm}\om^{2}_{M}\a^{-1}_{M}(m)\tr
 \b^{-1}_{A}\om_{A}(r^{1})\b^{-1}_{A}(\bar{r}^{1})
 \o\b_{A}(\bar{r}^{2})\o\a_{A}\psi_{A}(r^{2})\\ 
 \hspace{-3mm}&\stackrel{(\mref{eq:1.15})}=&\hspace{-3mm}(\om^{2}_{M}\a^{-2}_{M}(m)\tr\om_{A}\a^{-1}_{A}(r^{1}))\tr\bar{r}^{1}
 \o\b_{A}(\bar{r}^{2})\o\psi_{A}\b_{A}(r^{2})\\
 \hspace{-3mm}&=&\hspace{-3mm}(\widetilde{\varphi}\o\psi_{A})\widetilde{\varphi}(m).
 \end{eqnarray*}
 Then we verify the compatibility condition. For all $a\in A$ and $m\in M$,
 \begin{eqnarray*}
 &&\hspace{-20mm}\om_{M}(m)\tr a_{1}\o \b_{A}(a_{2})+\a_{M}(m_{(0)})\o m_{(1)}\psi_{A}(a)+\l\a_{M}\om_{M}(m)\o \b_{A}\psi_{A}(a)\\
 &\stackrel{(\mref{eq:14.25})}=&\hspace{-5mm}
 \om_{M}(m)\tr \om_{A}\a^{-1}_{A}(a)r^{1}\o \b^{2}_{A}(r^{2})-\om_{M}(m)\tr\a_{A}(r^{1})\o \b_{A}(r^{2})\psi_{A}(a)\\
 &&\hspace{-5mm}-\l\om_{M}(m)\tr 1\o \b_{A}\psi_{A}(a)+\om_{M}(m)\tr\a_{A}(r^{1})\o\b_{A}(r^{2})\psi_{A}(a)\\
 \hspace{-5mm}&&\hspace{-5mm}+\l\a_{M}\om_{M}(m)\o \b_{A}\psi_{A}(a)\\
 &\stackrel{(\mref{eq:1.5})(\mref{eq:1.15})}=&\hspace{-3mm}\om_{M}(m)\tr \om_{A}\a^{-1}_{A}(a)\b^{-1}(r^{1})\o \b_{A}(r^{2})\\
 &=&\hspace{-5mm}(m\tr a)_{(0)}\o (m\tr a)_{(1)},
 \end{eqnarray*}
 finishing the proof.
 \end{proof}
\vspace{-3mm}

 \begin{pro}\mlabel{pro:20.10} Let $(A,\mu,1,\a_{A},\b_{A},\psi_{A},\om_{A})$ be a quasitriangular (resp.  anti-quasitriangular) unitary $\l$-infBH-bialgebra, $(M,\gamma,\nu,\a_{M},\b_{M})$ (resp. $(M, \widetilde{\gamma}, \widetilde{\nu}, \a_{M}, \b_{M})$) be an $(A,\mu,\a_{A},\b_{A})$-bimodule. Then $(M,\gamma,\nu,\rho,\varphi,\a_{M},\b_{M},\psi_{M},\om_{M})$ (resp.  $(M,\widetilde{\gamma},\widetilde{\nu},\widetilde{\rho},\widetilde{\varphi},\a_{M},\b_{M},\psi_{M},\om_{M})$) is a $\l$-infBH-Hopf bimodule, where $\rho$ and $\varphi$ (resp.  $\widetilde{\rho}$ and $\widetilde{\varphi}$) are defined in Theorem \mref{thm:12.02} and \mref{thm:13.001} (resp. Theorem \mref{thm:12.02a} and \mref{thm:13.002a}).
 \end{pro}

 \begin{proof} We only prove the case for quasitriangular unitary $\l$-infBH-bialgebra. By Theorem \mref{thm:12.02} and \mref{thm:13.001}, we know that $(M,\gamma, \rho, \a_{M},\b_{M},\psi_{M},\om_{M})$ is a left $\l$-infBH-Hopf module and $(M, \nu, \varphi$, $\a_{M}$, $\b_{M},\psi_{M},\om_{M})$ is a right $\l$-infBH-Hopf module. Moreover,  for all $m\in M$, we have
 \begin{eqnarray*}
 (\om_{A}\o\varphi)\rho(m)
 &=&-\a_{A}\om_{A}(r^{1})\o(\om_{A}\a_{A}^{-1}(r^{2})
 \tl\a_{M}^{-1}\b_{M}^{-1}\psi_{M}\om_{M}(m))
 \tr \bar{r}^{1}\o\b_{A}(\bar{r}^{2})\\
 &&+\l\a_{A}\om_{A}(r^{1})\o(\om_{A}(r^{2})\tl \psi_{M}\om_{M}\b_{M}^{-1}(m))\o 1\\
 &\stackrel{(\mref{eq:1.16})}=&\rho(\om_{M}\a_{M}^{-1}(m)\tr r^{1})\o\psi_{A}\b_{A}(r^{2})
 -\l\rho(\om_{M}(m))\o 1\\
 &\stackrel{(\mref{eq:12.30})}=&(\rho\o\psi_{A})\varphi(m).
 \end{eqnarray*}
 Then $(M, \rho,\varphi, \psi_{M},\om_{M})$ is a $(A,\D_r, \psi_{A},\om_{A})$-bicomodule. Furthermore, by Eq.(\mref{eq:1.16}), we can get
 \begin{eqnarray*}
 (\gamma\o\b_{A})(\om_{A}\o\varphi)=\varphi\circ\gamma\quad \hbox{and}\quad (\a_{A}\o\nu)(\rho\o\psi_{A})=\rho\circ\nu.
 \end{eqnarray*}
 Hence, by Definition \mref{de:20.01}, we finish the proof.
 \end{proof} 

\vspace{-3mm}
 \begin{defi}\mlabel{de:12.03}(\cite{MM}) Let $(C,\D,\v,\psi,\omega)$ be a counitary BiHom-coassociative coalgebra, $\alpha,\beta: C\longrightarrow C$ be linear maps and $\sigma \in (C\otimes C)^{\ast}$. We call
 \begin{eqnarray}
 &\sigma (\alpha(c_{1}),\beta\omega (e))\sigma (c_{2},\psi (d))-\sigma (\omega(c),d_{1})\sigma (d_{2},\psi (e))\qquad\qquad\qquad\qquad\qquad&\nonumber\\
 &\qquad\qquad\qquad+\sigma(\omega(d),e_{1})\sigma (\alpha\psi(c),\beta(e_{2}))=\lambda(\hbox{{\bf resp.}}\ (-\l))\sigma(\alpha(c),\beta(e))\varepsilon(d)\qquad&\mlabel{eq:01.03}
 \end{eqnarray}
 the {\bf $\lambda$(\hbox{{\bf resp.}}\ ($-\l$))-coassociative BiHom-Yang-Baxter equation (abbr. $\lambda$(\hbox{{\bf resp.}}\ ($-\l$))-coabhYBe) in $(C,\D,\v,\psi,\omega)$} where $\lambda$ is a given element in $K$.
 \end{defi}
\vspace{-3mm}

 \begin{defi}\mlabel{de:13.004} (\cite{MM}) Let $(C,\Delta,\varepsilon,\psi,\omega)$ be a counitary BiHom-coassociative coalgebra such that $\psi,\omega$ are bijective, $\alpha,\beta: A\longrightarrow A$ be linear maps, $\sigma\in (C\otimes C)^{\ast}$ be $\alpha,\beta,\psi,\omega$-invariant and moreover Eqs.(\mref{eq:12.1}), (\mref{eq:12.3}) and (\mref{eq:12.31}) hold.  A {\bf coquasitriangular (\hbox{{\bf resp.}}\ anti-coquasi triangular) counitary $\l$-infBH-bialgebra} is a 8-tuple $(C, \Delta, \varepsilon, \alpha, \beta, \psi, \omega, \sigma)$ consisting of a counitary BiHom-coassociative coalgebra $(C, \Delta, \varepsilon, \psi, \omega)$ and a solution $\sigma\in (C\otimes C)^{\ast}$ of a $\l$(\hbox{{ resp.}}\ $(-\l)$)-coabhYBe, in this case, the multiplication $\mu_{\sigma}$(\hbox{{resp.}}\  $\widetilde{\mu_{\sigma}}$) is defined by \begin{eqnarray}
 &\mu_{\sigma}(c\otimes d)=\alpha\omega^{-1}(c_{1})\sigma(c_{2},\psi (d))-\sigma(\omega (c),d_{1})\beta\psi^{-1}(d_{2})-\lambda\alpha(c)\varepsilon(d)&\mlabel{eq:01.04}
 \end{eqnarray}
 ({\bf resp.} \begin{eqnarray}
 &\widetilde{\mu_{\sigma}}(c\otimes d)=\alpha\omega^{-1}(c_{1})\sigma(c_{2},\psi (d))-\sigma(\omega (c),d_{1})\beta\psi^{-1}(d_{2})-\lambda\varepsilon(c)\beta(d).)&\mlabel{eq:01.08}
 \end{eqnarray}
 \end{defi}
\vspace{-3mm}

 \begin{pro}\mlabel{pro:12.08} (\cite{MM}) With the assumptions of Definition \mref{de:13.004}, $(C, \mu=\mu_{\sigma}(\hbox{{ resp.}}\  \widetilde{\mu_{\sigma}})$, $\Delta, \varepsilon, \alpha, \beta, \psi, \omega)$, where $\mu_{\sigma}(\hbox{{resp.}}\  \widetilde{\mu_{\sigma}})$ is defined by Eq.(\mref{eq:01.04})(\hbox{{resp.}}\  Eq.(\mref{eq:01.08})), is a coquasitriangular (\hbox{{resp.}}\  anti-coquasitriangular) counitary $\l$-infBH-bialgebra if and only if
 \begin{eqnarray}
 &\sigma(cd,\beta(e))=-\sigma(\omega(d),e_{1})\sigma(\alpha\psi(c),\beta(e_{2}))& \mlabel{eq:01.06}
 \end{eqnarray}
 or
 \begin{eqnarray}
 &\sigma(\alpha(c),de)=\sigma(\alpha(c_{1}),\beta\omega(e))\sigma(c_{2},\psi(d))
  -\lambda\sigma(\alpha(c),\beta(e))\varepsilon(d)-\lambda\sigma(c,d)\varepsilon(e)& \mlabel{eq:01.07}
 \end{eqnarray}
 (\hbox{{resp.}}
  \begin{eqnarray}
 &\sigma(cd,\beta(e))=-\sigma(\omega(d),e_{1})\sigma(\alpha\psi(c),\beta(e_{2}))
 -\lambda\sigma(\alpha(c),\beta(e))\varepsilon(d)-\lambda\sigma(d,e)\varepsilon(c)& \mlabel{eq:01.10}
 \end{eqnarray}
 or
 \begin{eqnarray}
  &\sigma(\alpha(c),de)=\sigma(\alpha(c_{1}),\beta\omega(e))\sigma(c_{2},\psi(d))& \mlabel{eq:01.11})
 \end{eqnarray}
 hold for all $c, d, e\in C$.
 \end{pro} 
 
 The following four theorems provide four different constructions of left (right) $\l$-infBH-Hopf module from the comodule of (anti-)coquasitriangular counitary $\l$-infBH-bialgebra.
\vspace{-3mm}

 \begin{thm}\mlabel{thm:12.013}(\cite{MM}) Let $(C, \D, \v, \a_{C}, \b_{C}, \psi_{C}, \om_{C}, \sigma)$ be a coquasitriangular counitary $\l$-infBH-bialgebra and $(M,\rho,\psi_{M},\om_{M})$ be a left $(C,\D,\psi_{C},\om_{C})$-comodule, $\a_{M},\b_{M}: M\longrightarrow M$ be linear maps such that $\b_{M}\ci \psi_{M}=\psi_{M}\ci \b_{M}$, $\rho\ci \b_{M}=(\b_{A}\o \b_{M})\ci \rho$. Then $(M, \gamma, \rho, \a_{M}, \b_{M}, \psi_{M}, \om_{M})$ becomes a left $\l$-infBH-Hopf module with the action $\gamma: C\o M\longrightarrow M$ given by
 \begin{eqnarray*}
 &\gamma(c\o m):=-\sigma(\om_{C}(c),m_{-1})\b_{M}\psi_{M}^{-1}(m_{0})\ for\ c\in C, \ m\in M.&
 \end{eqnarray*}
 \end{thm}
\vspace{-3mm}

 \begin{thm}\mlabel{thm:13.014} Let $(C, \D, \v, \a_{C}, \b_{C}, \psi_{C}, \om_{C}, \sigma)$ be a coquasitriangular counitary $\l$-infBH-bialgebra and $(M, \varphi, \psi_{M}, \om_{M})$ be a right $(C, \D, \psi_{C}, \om_{C})$-comodule, $\a_{M}, \b_{M}: M\longrightarrow M$ be linear maps such that $\a_{M}\ci \om_{M}=\om_{M}\ci \a_{M}$, $\vp\ci \a_{M}=(\a_{M}\o \a_{A})\ci \vp$. Then $(M, \nu, \varphi, \a_{M}, \b_{M}, \psi_{M}, \om_{M})$ becomes a right $\l$-infBH-Hopf module with the action $\nu: M\o C\longrightarrow M$ given by
 \begin{eqnarray*}
 &\nu(m\o c):=\a_{M}\om^{-1}_{M}(m_{(0)})\sigma(m_{(1)},\psi_{C}(c))-\l\v(c)\a_{M}(m), \ for\ c\in C, \ m\in M.&
 \end{eqnarray*}
 \end{thm}

 \begin{proof} For all $c, d\in C$ and $m\in M$, one calculates
 \begin{eqnarray*}
 \nu(\nu\o\b_{C})(m\o c\o d)
 &\stackrel{(\mref{eq:12.31})}=&
 \sigma(m_{(1)},\psi_{C}(c))\sigma(\a_{C}\om^{-1}_{C}(m_{(0)(1)}),\b_{C}\psi_{C}(d))
 \a_{M}^{2}\om_{M}^{-2}(m_{(0)(0)})\\
 &&-\l\v(c)\sigma(\a_{C}(m_{(1)}),\psi_{C}\b_{C}(d))
 \a_{M}^{2}\om_{M}^{-1}(m_{(0)})
 +\l^{2}\v(c)\v(d)\a_{M}^{2}(m)\\
 &&-\l\v(d)\sigma(m_{(1)},\psi_{C}(c))\a_{M}^{2}\om_{M}^{-1}(m_{(0)})\\
 &\stackrel{(\mref{eq:01.07})(\mref{eq:1.11})}=&\nu(\a_{M}\o \mu_{\sigma})(m\o c\o d).
 \end{eqnarray*}
 Thus $(M, \varphi, \a_{M}, \b_{M})$ is a right $(C, \mu, \a_{C}, \b_{C})$-module. Next we check the compatibility condition.
 \begin{eqnarray*}
 &&\hspace{-18mm}\om_{M}(m)\tr c_{1}\o \b_{C}(c_{2})+\a_{M}(m_{(0)})\o m_{(1)}\psi_{C}(c)+\l\a_{M}\om_{M}(m)\o \b_{C}\psi_{C}(c)\\
 &\stackrel{(\mref{eq:01.04})(\mref{eq:1.11})}=&\sigma(\om_{C}(m_{(1)}),\psi_{C}(c_{1}))
 \a_{M}(m_{(0)})\o\b_{C}(c_{2})
 -\l\v(c_{1})\a_{M}\om_{M}(m)\o\b_{C}(c_{2})\\
 &&+\sigma(m_{(1)2},\psi_{C}^{2}(c))\a_{M}(m_{(0)})\o\a_{C}\om_{C}^{-1}(m_{(1)1})
 -\sigma(\om_{C}(m_{(1)}),\psi_{C}(c_{1}))\a_{M}(m_{(0)})\o\b_{C}(c_{2})\\
 &&-\l\v(c)\a_{M}(m_{(0)})\o \a_{C}(m_{(1)})
 +\l\a_{M}\om_{M}(m)\o\b_{C}\psi_{C}(c_{2})\\
 &\stackrel{(\mref{eq:1.11})}=&\sigma(m_{(1)2},\psi_{C}^{2}(c))
 \a_{M}(m_{(0)})\o\a_{C}\om_{C}^{-1}(m_{(1)1})
 -\l\v(c)\a_{M}(m_{(0)})\o \a_{C}(m_{(1)})\\
 &=&(m\tr c)_{(0)}\o (m\tr c)_{(1)},
 \end{eqnarray*}
 completing the proof.
 \end{proof}
\vspace{-3mm}

 \begin{thm}\mlabel{thm:13.013}
 Let $(C,\D,\v,\a_{C},\b_{C},\psi_{C},\om_{C},\sigma)$ be an anti-coquasitriangular counital $\l$-infBH-bialgebra and $(M,\widetilde{\rho},\psi_{M},\om_{M})$ be a left $(C,\D,\psi_{C},\om_{C})$-comodule, $\a_{M},\b_{M}:M\longrightarrow M$ be linear maps such that $\b_{M}\ci \psi_{M}=\psi_{M}\ci \b_{M}$, $\widetilde{\rho}\ci \b_{M}=(\b_{A}\o \b_{M})\ci \widetilde{\rho}$. Then $(M,\widetilde{\gamma},\widetilde{\rho},\a_{M},\b_{M},\psi_{M},\om_{M})$ becomes a left $\l$-infBH-Hopf module with the action $\widetilde{\gamma}: C \o M\longrightarrow M$ given by
 \begin{eqnarray*}
 &\widetilde{\gamma}(c\o m):=-\sigma(\om_{C}(c),m_{-1})\b_{M}\psi_{M}^{-1}(m_{0})-\l\v(c)\b_{M}(m)\ for\ c\in C, \ m\in M.&
 \end{eqnarray*}
 \end{thm}
\vspace{-3mm}

 \begin{thm}\mlabel{thm:13.015} Let $(C, \D, \v, \a_{C}, \b_{C}, \psi_{C}, \om_{C}, \sigma)$ be an anti-coquasitriangular counital $\l$-infBH-bialgebra and $(M,\widetilde{\varphi},\psi_{M},\om_{M})$ be a right $(C,\D,\psi_{C},\om_{C})$-comodule, $\a_{M},\b_{M}:M\longrightarrow M$ be linear maps such that $\a_{M}\ci \om_{M}=\om_{M}\ci \a_{M}$, $\widetilde{\varphi}\ci \a_{M}=(\a_{M}\o \a_{A})\ci \widetilde{\varphi}$. Then $(M, \widetilde{\nu}, \widetilde{\varphi}, \a_{M}, \b_{M}, \psi_{M}, \om_{M})$ becomes a right $\l$-infBH-Hopf module with the action $\widetilde{\nu}: M\o C\longrightarrow M$ given by
 \begin{eqnarray*}
 &\widetilde{\nu}(m\o c):=\a_{M}\om^{-1}_{M}(m_{(0)})\sigma(m_{(1)},\psi_{C}(c)), \ for\ c\in C, \ m\in M.&
 \end{eqnarray*}
 \end{thm}

 \begin{proof} Similar to Theorem \mref{thm:13.014}.   
 \end{proof}
\vspace{-3mm}
 
 \begin{pro}\mlabel{pro:C20.10} Let $(C,\D,\v,\a_{C},\b_{C},\psi_{C},\om_{C})$ be a coquasitriangular (or anti-coquasitriangular) counitary $\l$-infBH bialgebra, $(M,\rho,\varphi,\psi_{M},\om_{M})$ (or $(M,\widetilde{\rho},\widetilde{\varphi},\psi_{M},\om_{M})$) be a $(C,\D,\psi_{C},\om_{C})$-bicomodule. Then $(M,\gamma,\nu,\rho,\varphi,\a_{M},\b_{M},\psi_{M},\om_{M})$ (or $(M,\widetilde{\gamma},\widetilde{\nu},\widetilde{\rho},\widetilde{\varphi},\a_{M},\b_{M},\psi_{M},\om_{M})$) is a $\l$-infBH-Hopf bimodule, where $\gamma$ and $\nu$ (or $\widetilde{\gamma}$ and $\widetilde{\nu}$) are defined by
 \begin{eqnarray*}
 &\gamma(c\o m):=-\sigma(\om_{C}(c),m_{-1})\b_{M}\psi_{M}^{-1}(m_{0}),&\\
 &\nu(m\o c):=\a_{M}\om^{-1}_{M}(m_{(0)})\sigma(m_{(1)},\psi_{C}(c))-\l\v(c)\a_{M}(m).&
 \end{eqnarray*}
 (or
 \begin{eqnarray*}
 &\widetilde{\gamma}(c\o m):=-\sigma(\om_{C}(c),m_{-1})\b_{M}\psi_{M}^{-1}(m_{0})-\l\v(c)\b_{M}(m),&\\
 &\widetilde{\nu}(m\o c):=\a_{M}\om^{-1}_{M}(m_{(0)})\sigma(m_{(1)},\psi_{C}(c)),&
 \end{eqnarray*}) respectively.
 \end{pro}
\vspace{-3mm}

\section{General Gelfand-Dorfman Theorem on BiHom-Novikov algebras}
 \begin{defi}\mlabel{de:13.4} (\cite[Definition 2.1]{LMMP3}) Let $(A,\cdot,\a,\b)$ be a (left) BiHom-pre-Lie algebra. Then the 4-tuple $(A,\cdot,\a,\b)$ is called a {\bf BiHom-Novikov algebra} if
 \begin{eqnarray}
 &(a\cdot\b(b))\cdot\a\b(c)=(a\cdot\b(c))\cdot\a\b(b), \quad \forall a, b, c\in A.&\mlabel{eq:13.03}
 \end{eqnarray}
 \end{defi}
\vspace{-3mm}

 \begin{defi}\mlabel{de:13.5} (\cite[Definition 2.4]{LMMP3}) A BiHom-associative algebra $(A,\mu,\a,\b)$ is called {\bf BiHom-commutative} if
 \begin{eqnarray}
 &\b(a)\a(b)=\b(b)\a(a),  \quad \forall a, b \in A.&\mlabel{eq:13.3}
 \end{eqnarray}
 \end{defi}
\vspace{-3mm}

 \begin{defi}\mlabel{de:20.14} (\cite[Definition 3.1]{LMMP}) Let $(A,\mu,\a,\b)$ be a BiHom-associative algebra, $\xi,\zeta: A\longrightarrow A$ two algebra maps and $\mathfrak{D}: A\longrightarrow A$ be a linear map. We call $\mathfrak{D}$ a {\bf $\l$-$(\xi,\zeta)$-derivation} if
 \begin{eqnarray}\mlabel{eq:deri}
 \mathfrak{D}(ab)=\xi(a)\mathfrak{D}(b)
 +\mathfrak{D}(a)\zeta(b)+\l\xi(a)\zeta(b), \quad \forall a,b\in A.
 \end{eqnarray}
 \end{defi}
\vspace{-3mm}

 \begin{lem}\mlabel{lem:13.6} Let $(A,\mu,\D,\a,\b,\psi,\om)$ be a $\l$-infBH-bialgebra. Then $\mathfrak{D}=\mu\D$ is a $\l$-$(\a\om,\b\psi)$-derivation, i.e.
 \begin{eqnarray*}
 &\mathfrak{D}(ab)=\a\om(a)\mathfrak{D}(b)+\mathfrak{D}(a)\b\psi(b)
 +\l\a\om(a)\b\psi(b).&
 \end{eqnarray*}
 \end{lem}

 \begin{proof} For all $a, b\in A$, we have
 \begin{eqnarray*}
 \mathfrak{D}(ab)&=&\mu\D(ab)\\
 &\stackrel{(\mref{eq:12.4})}=&(\om(a)b_{1})\b(b_{2})+\a(a_{1})(a_{2}\psi(b))+ \l\a\om(a)\b\psi(b)\\
 &\stackrel{(\mref{eq:1.3})}=&\a\om(a)(b_{1} b_{2})+(a_{1} a_{2})\b\psi(b)+ \l\a\om(a)\b\psi(b)\\
 &=&\a\om(a)\mathfrak{D}(b)+\mathfrak{D}(a)\b\psi(b)+ \l\a\om(a)\b\psi(b),
 \end{eqnarray*}
 as required.
 \end{proof}

\subsection{BiHom-type of Gelfand-Dorfman Theorem: Approach 1}
 \begin{thm}\mlabel{thm:20.15} Let $(A,\mu,\a,\b)$ be a commutative BiHom-associative algebra, $\xi,\zeta,\iota:A\longrightarrow A$ three algebra maps and $\mathfrak{D}$ a $\l$-$(\xi,\xi)$-derivation such that any two of the maps $\a,\b,\xi,\zeta,\iota,\mathfrak{D}$ commute. Define a new multiplication on $A$ by
 $$
 a\ast b=\zeta(a)\iota\mathfrak{D}(b),\quad \forall a,b\in A.
 $$
 Then $(A,\ast,\zeta\a,\iota\b\xi)$ is a BiHom-pre-Lie algebra and a BiHom-Novikov algebra.
 \end{thm}

 \begin{proof} It is easy to see that $\zeta\a(a\ast b)=\zeta\a(a)\ast\zeta\a(b)$ and $\iota\b\xi(a\ast b)=\iota\b\xi(a)\ast\iota\b\xi(b)$. Now, for all $a,b,c\in A$, we compute
 \begin{eqnarray*}
 &&\hspace{-20mm}\zeta\a\iota\b\xi(a)\ast(\zeta\a(b)\ast c)-(\iota\b\xi(a)\ast\zeta\a(b))\ast \iota\b\xi(c)\\
 &=&
 \zeta^{2}\a\iota\b\xi(a)(\iota\xi\zeta^{2}\a(b)\iota^{2}\mathfrak{D}^{2}(c))
 +\zeta^{2}\a\iota\b\xi(a)(\iota\zeta^{2}\a\mathfrak{D}(b)\xi\iota^{2}\mathfrak{D}(c))\\
 &&
 +\l\zeta^{2}\a\iota\b\xi(a)(\zeta^{2}\xi\iota\a(b)\iota^{2}\xi\mathfrak{D}(c))
 -(\zeta^{2}\iota\b\xi(a)\zeta^{2}\a\iota\mathfrak{D}(b))\iota^{2}\b\xi\mathfrak{D}(c)\\
 &\stackrel{(\mref{eq:1.3})}=&
 (\zeta^{2}\iota\b\xi(a)\iota\xi\zeta^{2}\a(b))\b\iota^{2}\mathfrak{D}^{2}(c)
 +(\zeta^{2}\iota\b\xi(a)\iota\zeta^{2}\a\mathfrak{D}(b))\b\xi\iota^{2}\mathfrak{D}(c)\\
 &&
 +\l(\zeta^{2}\iota\b\xi(a)\zeta^{2}\xi\iota\a(b))\b\iota^{2}\xi\mathfrak{D}(c)
 -(\zeta^{2}\iota\b\xi(a)\zeta^{2}\a\iota\mathfrak{D}(b))\iota^{2}\b\xi\mathfrak{D}(c)\\
 &=&
 (\zeta^{2}\iota\b\xi(a)\iota\xi\zeta^{2}\a(b))\b\iota^{2}\mathfrak{D}^{2}(c)
 +\l(\zeta^{2}\iota\b\xi(a)\zeta^{2}\xi\iota\a(b))\b\iota^{2}\xi\mathfrak{D}(c)\\
 &\stackrel{(\mref{eq:13.3})}=&
 (\zeta^{2}\iota\b\xi(b)\iota\xi\zeta^{2}\a(a))\b\iota^{2}\mathfrak{D}^{2}(c)
 +\l(\zeta^{2}\iota\b\xi(b)\zeta^{2}\xi\iota\a(a))\b\iota^{2}\xi\mathfrak{D}(c),
 \end{eqnarray*}
 which means that $(A,\ast,\zeta\a,\iota\b\xi)$ is a BiHom-pre-Lie algebra. Next we show that $(A,\ast,\zeta\a,\iota\b\xi)$ is further a BiHom-Novikov algebra. Note first that, since $(A,\mu,\a,\b)$ is BiHom-commutative, it is BiHom-Novikov, so Eq.(\mref{eq:13.03}) holds. Now we compute:
 \begin{eqnarray*}
 (a\ast \iota\b\xi(b))\ast\zeta\a\iota\b\xi(c)
 &=&(\zeta^{2}(a)\b(\zeta\mathfrak{D}\iota^{2}\xi(b)))\a\b\mathfrak{D}(\iota^{2}\xi\zeta(c))\\
 &\stackrel{(\mref{eq:13.03})}=&
 (\zeta^{2}(a)\b(\zeta\mathfrak{D}\iota^{2}\xi(c)))\a\b\mathfrak{D}(\iota^{2}\xi\zeta(b))\\
 &=&(a\ast \iota\b\xi(c))\ast\zeta\a\iota\b\xi(b).
 \end{eqnarray*}
 This completes the proof.
 \end{proof} 
 
 \begin{rmk} If $\l=0$ in Theorem \mref{thm:20.15}, then we obtain \cite[Proposition 2.7]{LMMP3}.
 \end{rmk}

\subsection{BiHom-type of Gelfand-Dorfman Theorem: Approach 2}
 
 \begin{thm}\mlabel{thm:B20.15} Let $\kappa, \l$ be two given elements of $K$, $(A,\mu,\a,\b)$ be a commutative BiHom-associative algebra, $\xi:A\longrightarrow A$ be an algebra map and $\mathfrak{D}$ be a $\l$-$(\xi,\xi)$-derivation such that any two of the maps $\a,\b,\xi,\mathfrak{D}$ commute. Define a new multiplication on $A$ by
 $$
 a\ast b=\xi(a)\mathfrak{D}(b)+\kappa\xi(a)\xi(b),\quad \forall a,b\in A.
 $$
 Then $(A,\ast,\a\xi,\b\xi)$ is a BiHom-pre-Lie algebra and a BiHom-Novikov algebra.
 \end{thm}

 \begin{proof} It is easy to see that $\a\xi(a\ast b)=\a\xi(a)\ast\a\xi(b)$ and $\b\xi(a\ast b)=\b\xi(a)\ast\b\xi(b)$. Now, for all $a,b,c\in A$, we compute
 \begin{eqnarray*}
 &&\a\b\xi^{2}(a)\ast(\a\xi(b)\ast c)-(\b\xi(a)\ast\a\xi(b))\ast \b\xi(c)\\
 &&\hspace{10mm}=\a\b\xi^{3}(a)\mathfrak{D}(\a\xi^{2}(b)\mathfrak{D}(c)
 +\kappa\a\xi^{2}(b)\xi(c))+\kappa\a\b\xi^{3}(a)(\a\xi^{3}(b)\xi\mathfrak{D}(c)
 +\kappa\a\xi^{3}(b)\xi^{2}(c))\\
 &&\hspace{14mm}-(\b\xi^{3}(a)\a\xi^{2}\mathfrak{D}(b)
 +\kappa\b\xi^{3}(a)\a\xi^{3}(b))
 \b\xi\mathfrak{D}(c)-\kappa(\b\xi^{3}(a)\a\xi^{2}\mathfrak{D}(b)\\
 &&\hspace{14mm}+\kappa\b\xi^{3}(a)\a\xi^{3}(b))\b\xi^{2}(c)\\
 &&\hspace{10mm}=\a\b\xi^{3}(a)(\a\xi^{3}(b)\mathfrak{D}^{2}(c))
 +\a\b\xi^{3}(a)(\a\xi^{2}\mathfrak{D}(b)\xi\mathfrak{D}(c)) +\l\a\b\xi^{3}(a)(\a\xi^{3}(b)\xi\mathfrak{D}(c))\\
 &&\hspace{14mm}+\kappa\a\b\xi^{3}(a)(\a\xi^{3}(b)\xi\mathfrak{D}(c)) +\kappa\a\b\xi^{3}(a)(\a\xi^{2}\mathfrak{D}(b)\xi^{2}(c))
 +\l\kappa\a\b\xi^{3}(a)(\a\xi^{3}(b)\xi^{2}(c))\\
 &&\hspace{14mm}+\kappa(\b\xi^{3}(a)\a\xi^{3}(b))\xi\b\mathfrak{D}(c)
 +\kappa^{2}(\b\xi^{3}(a)\a\xi^{3}(b))\b\xi^{2}(c) -(\b\xi^{3}(a)\a\xi^{2}\mathfrak{D}(b))\b\xi\mathfrak{D}(c)\\
 &&\hspace{14mm}-\kappa(\b\xi^{3}(a)\a\xi^{3}(b))\b\xi\mathfrak{D}(c) -\kappa(\b\xi^{3}(a)\a\xi^{2}\mathfrak{D}(b))\b\xi^{2}(c)
 -\kappa^{2}(\b\xi^{3}(a)\a\xi^{3}(b))\b\xi^{2}(c)\\
 &&\hspace{9mm}\stackrel{(\mref{eq:1.3})}=(\b\xi^{3}(a)\a\xi^{3}(b))\b\mathfrak{D}^{2}(c)
 +\l(\b\xi^{3}(a)\a\xi^{3}(b))\b\xi\mathfrak{D}(c) +\l\kappa(\b\xi^{3}(a)\a\xi^{3}(b))\b\xi^{2}(c)\\
 &&\hspace{14mm}+\kappa(\b\xi^{3}(a)\a\xi^{3}(b))\xi\b\mathfrak{D}(c)\\
 &&\hspace{9mm}\stackrel{(\mref{eq:13.3})}=(\b\xi^{3}(b)\a\xi^{3}(a))\b\mathfrak{D}^{2}(c)
 +\l(\b\xi^{3}(b)\a\xi^{3}(a))\b\xi\mathfrak{D}(c)+\l\kappa(\b\xi^{3}(b)\a\xi^{3}(a))\b\xi^{2}(c)\\
 &&\hspace{14mm}+\kappa(\b\xi^{3}(b)\a\xi^{3}(a))\xi\b\mathfrak{D}(c)\\
 &&\hspace{10mm}=\a\b\xi^{2}(b)\ast(\a\xi(a)\ast c)-(\b\xi(b)\ast\a\xi(a))\ast \b\xi(c),
 \end{eqnarray*}
 which means that $(A,\ast,\a\xi,\b\xi)$ is a BiHom-pre-Lie algebra. Furthermore, $(A,\ast,\a\xi,\b\xi)$ is a BiHom-Novikov algebra. Since $(A,\mu,\a,\b)$ is BiHom-commutative, it is BiHom-Novikov, so Eq.(\mref{eq:13.03}) holds.
 \begin{eqnarray*}
 (a\ast \b\xi(b))\ast\a\b\xi^{2}(c)
 &=&(\xi^{2}(a)\b\xi^{2}\mathfrak{D}(b))\a\b\xi^{2}\mathfrak{D}(c)
 +\kappa(\xi^{2}(a)\b\xi^{3}(b))\a\b\xi^{2}\mathfrak{D}(c)\\
 &&+\kappa(\xi^{2}(a)\b\xi^{2}\mathfrak{D}(b))\a\b\xi^{3}(c)
 +\kappa^{2}(\xi^{2}(a)\b\xi^{3}(b))\a\b\xi^{3}(c)\\
 &\stackrel{(\mref{eq:13.03})}=&(\xi^{2}(a)\b\xi^{2}\mathfrak{D}(c))\a\b\xi^{2}\mathfrak{D}(b)
 +\kappa(\xi^{2}(a)\b\xi^{2}\mathfrak{D}(c))\a\b\xi^{3}(b)\\
 &&+\kappa(\xi^{2}(a)\b\xi^{3}(c))\a\b\xi^{2}\mathfrak{D}(b)
 +\kappa^{2}(\xi^{2}(a)\b\xi^{3}(c))\a\b\xi^{3}(b)\\
 &=&(a\ast\b\xi(c))\ast\a\b\xi^{2}(b),
 \end{eqnarray*}
 finishing the proof.
 \end{proof}
\vspace{-3mm} 

 \begin{cor}\mlabel{cor:B013.9} Let $\kappa,\l$ be two given elements of $K$, $(A,\mu,\a,\b)$ be a commutative BiHom-associative algebra and $\mathfrak{D}$  be a $\l$-$(\a\b,\a\b)$-derivation such that any two of the maps $\a, \b, \mathfrak{D}$ commute. Define
 \begin{eqnarray*}
 &a\ast b=\a\b(a)\mathfrak{D}(b)+\kappa\a\b(a)\a\b(b),\quad \forall a,b\in A&
 \end{eqnarray*}
 Then $(A,\ast,\a^{2}\b,\a\b^{2})$ is a BiHom-pre-Lie algebra and a BiHom-Novikov algebra.
 \end{cor}

 \begin{proof} Let $\xi=\a\b$ in Theorem \mref{thm:B20.15}.
 \end{proof}
\vspace{-3mm}

 \begin{cor}\mlabel{cor:B113.A19a} Let $\kappa,\l$ be two given elements of $K$, $(A,\mu,\a,\b)$ be a commutative BiHom-associative algebra, $\mathfrak{D}:A\longrightarrow A$ be a $\l$-$(\id,\id)$-derivation such that any two of the maps $\a,\b,\mathfrak{D}$ commute. Define
 \begin{eqnarray*}
 &a\ast b=a\mathfrak{D}(b)+\kappa ab,\quad \forall a,b\in A&
 \end{eqnarray*}
 Then $(A,\ast,\a,\b)$ is a BiHom-pre-Lie algebra and a BiHom-Novikov algebra.
 \end{cor}

 \begin{proof} Let $\xi=\id$ in Theorem \mref{thm:B20.15}.
 \end{proof}
\vspace{-3mm}

 \begin{cor}\mlabel{cor:B013.9a} Let $\kappa,\l$ be two given elements of $K$, $(A,\mu,\D,\a,\b,\a,\b)$ be a BiHom-commutative $\l$-infBH-bialgebra, $\xi:A\longrightarrow A$ be an algebra map and coalgebra map such that any two of the maps $\a,\b,\xi$ commute. Define
 \begin{eqnarray*}
 &a\ast b=\xi(a)(b_1b_2)+\kappa\xi(a)\xi(b),\quad \forall a,b\in A&
 \end{eqnarray*}
 Then $(A,\ast,\a\xi,\b\xi)$ is a BiHom-pre-Lie algebra and a BiHom-Novikov algebra.
 \end{cor}

 \begin{proof} It can be obtained by Lemma \mref{lem:13.6} and Theorem \mref{thm:B20.15}.
 \end{proof}
\vspace{-3mm}

 \begin{cor}\mlabel{cor:B213.A29a} Let $(A,\mu)$ be an associative and commutative algebra, $\xi:A\longrightarrow A$ be an algebra map and $\mathfrak{D}:A\longrightarrow A$ be  a $\l$-$(\xi,\xi)$-derivation. Assume that moreover we have $\mathfrak{D}\circ\xi=\xi\circ\mathfrak{D}$. Define a new multiplication on $A$ by
 \begin{eqnarray*}
 &a\ast b=\xi(a)\mathfrak{D}(b)+\kappa\xi(a)\xi(b),\quad \forall a,b\in A&
 \end{eqnarray*}
 Then $(A,\ast,\xi,\xi)$ is a BiHom-pre-Lie algebra and a BiHom-Novikov algebra.
 \end{cor}

 \begin{proof} Let $\a=\b=\id$ in Theorem \mref{thm:B20.15}.
 \end{proof}
\vspace{-3mm}

 \begin{cor}\mlabel{cor:BB013.9a} Let $(A,\mu,\a,\b,\a,\b,r)$ be a quasitriangular (resp.  anti-quasitriangular) $\l$-infBH-bialgebra such that $\a,\b$ are bijective, $\xi:A\longrightarrow A$ be an algebra map such that any two of the maps $\a,\b,\xi$ commute. Define
 \begin{eqnarray*}
 &a\ast b=\xi(a)(\b(b)(r^{1}r^{2}))-\xi(a)((r^{1}r^{2})\a(b))
 -\l\xi(a)\a\b(b)+\kappa\xi(a)\xi(b),\quad \forall a,b\in A&
 \end{eqnarray*}
 Then $(A,\ast,\a\xi,\b\xi)$ is a BiHom-pre-Lie algebra and a BiHom-Novikov algebra.
 \end{cor}

 \begin{proof} It can be obtained by Lemma \mref{lem:13.6} and Theorem \mref{thm:B20.15}.
 \end{proof}
\vspace{-3mm}

 \begin{defi}\mlabel{de:B022.008} (\cite[Definition 3.1]{LMMP3}) {\bf  A BiHom-Novikov-Poisson algebra} is a 5-tuple $(A,\cdot,\ast, \a,\b)$ such that:

 (1) $(A,\cdot,\a,\b)$ is a BiHom-commutative algebra;

 (2) $(A,\ast,\a,\b)$ is a BiHom-Novikov algebra;

 (3) the following compatibility conditions hold for all $a,b,c\in A$:
 \begin{eqnarray}
 &(\b(a)\ast\a(b))\cdot\b(c)-\a\b(a)\ast(\a(b)\cdot c)
 =(\b(b)\ast\a(a))\cdot\b(c)-\a\b(b)\ast(\a(a)\cdot c),&\mlabel{eq:B022.022}\\
 &(a\cdot\b(b))\ast \a\b(c)=(a\ast\b(c))\cdot\a\b(b),&\mlabel{eq:B022.023}\\
 &\a(a)\cdot(b\ast c)=(a\cdot b)\ast \b(c).&\mlabel{eq:B022.024}
 \end{eqnarray}
 \end{defi}
\vspace{-3mm}

 \begin{pro}\mlabel{pro:120.112} With  the hypotheses of Corollary \mref{cor:B113.A19a}, $(A,\mu,\ast,\a,\b)$ is a BiHom-Novikov-Poisson algebra.
 \end{pro}
\vspace{-3mm}

 \begin{proof} We only need to prove the relations (\mref{eq:B022.022}), (\mref{eq:B022.023}) and (\mref{eq:B022.024}) as follows. 
 \begin{eqnarray*}
 (\b(a)\ast\a(b)) \b(c)-\a\b(a)\ast(\a(b) c)
 \hspace{-2mm}&=&\hspace{-2mm}(\b(a)\a\mathfrak{D}(b)) \b(c)
 +\kappa(\b(a)\a(b)) \b(c)-\a\b(a)\mathfrak{D}(\a(b)  c)\\
 &&+\kappa\a\b(a)(\a(b)  c)\\
 \hspace{-2mm}&=&\hspace{-2mm}(\b(a)\a\mathfrak{D}(b)) \b(c)
 +\kappa(\b(a)\a(b)) \b(c)-\a\b(a) (\a(b) \mathfrak{D}(c))\\
 \hspace{-2mm}&&\hspace{-2mm}-\a\b(a) (\a\mathfrak{D}(b)  c)-\l\a\b(a) (\a(b)  c)
 +\kappa\a\b(a)(\a(b)  c)\\
 \hspace{-2mm}&\stackrel{(\mref{eq:1.3})}=&\hspace{-2mm}(\b(a)\a\mathfrak{D}(b)) \b(c)
 +\kappa(\b(a)\a(b)) \b(c)-(\b(a) \a(b)) \b\mathfrak{D}(c)\\
 \hspace{-2mm}&&\hspace{-2mm}-(\b(a) \a\mathfrak{D}(b)) \b(c)-\l(\b(a)\a(b)) \b(c)
 +\kappa(\b(a)\a(b)) \b(c)\\
 \hspace{-2mm}&=&\hspace{-2mm} 2\kappa(\b(a)\a(b)) \b(c)-(\b(a) \a(b)) \b\mathfrak{D}(c)
 -\l(\b(a)\a(b)) \b(c)\\
 \hspace{-2mm}&=&\hspace{-2mm}
 2\kappa(\b(b)\a(a)) \b(c)-(\b(b) \a(a)) \b\mathfrak{D}(c)
 -\l(\b(b)\a(a)) \b(c)\\
 \hspace{-2mm}&=&\hspace{-2mm}(\b(b)\ast\a(a)) \b(c)-\a\b(b)\ast(\a(a)  c),
 \end{eqnarray*} 
 \begin{eqnarray*}
 (a \b(b))\ast\a\b(c)
 &=&(a \b(b)) \a\b\mathfrak{D}(c)
 +\kappa(a \b(b)) \a\b(c)\\
 &\stackrel{(\mref{eq:1.3})}=&\a(a) (\b(b) \a\mathfrak{D}(c))
 +\kappa\a(a) (\b(b) \a(c))\\
 &\stackrel{(\mref{eq:13.3})}=&\a(a) (\b\mathfrak{D}(c) \a(b))
 +\kappa\a(a) (\b(c) \a(b))\\
 &\stackrel{(\mref{eq:1.3})}=&(a \b\mathfrak{D}(c)) \a\b(b)+\kappa(a \b(c)) \a\b(b)\\
 &=&(a\ast\b(c)) \a\b(b)
 \end{eqnarray*} 
 and
 \begin{eqnarray*}
 \a(a) (b\ast c)
 &=&\a(a) (b \mathfrak{D}(c))
 +\kappa\a(a) (b  c)\\
 &\stackrel{(\mref{eq:1.3})}=&(a  b) \b\mathfrak{D}(c)
 +\kappa (a   b) \b(c)=(a  b)\ast \b(c),
 \end{eqnarray*}
 finishing the proof.
 \end{proof} 
 
 \section*{Acknowledgment} Ma is supported by Natural Science Foundation of Henan Province (No.212300410365). 

 \end{document}